\documentclass[12pt]{article}
\usepackage[latin1]{inputenc}
\usepackage[british]{babel}
\usepackage{lmodern}
\usepackage[T1]{fontenc}
\usepackage[paper=a4paper, left=28mm, right=25mm, top=29mm, bottom=25mm]{geometry}
\usepackage{latexsym,amsfonts,amsmath,graphics}
\usepackage{epsfig}
\usepackage{enumitem}
\usepackage{amssymb,mathrsfs}
\usepackage{verbatim}
\usepackage{algpseudocode} 
\usepackage{algorithm} 
\usepackage{float}
\newtheorem{theorem}{Theorem}
\newtheorem{lemma}{Lemma}
\newtheorem{corollary}{Corollary}
\newtheorem{example}{Example}

\newtheorem{remark}{Remark}
\newtheorem{definition}{Definition}
\newenvironment{proof}{\begin{trivlist} \item[\hskip\labelsep{\it Proof.}]}{$\hfill\Box$\end{trivlist}}

\newcommand{\rd}{\,\mathrm{d}}

\newcommand{\bsx}{\boldsymbol{x}}

\newcommand{\RR}{\mathbb{R}}

\newcommand{\NN}{\mathbb{N}}

\newcommand{\cR}{\mathcal{R}}

\newcommand{\id}{\mathrm{id}}

\newcommand{\sym}{{\rm sym}}
\newcommand{\cP}{\mathcal{P}}

\newcommand{\osigma}{\overline{\sigma}}

\allowdisplaybreaks

\title{$L_2$ discrepancy of symmetrized generalized Hammersley point sets in base $b$}
\author{Ralph Kritzinger\thanks{R. Kritzinger is supported by the Austrian Science Fund (FWF): Project F5509-N26, which is a part of the Special Research Program "Quasi-Monte Carlo Methods: Theory and Applications".} \, and Lisa M. Kritzinger}
\date{}

\begin{document}

\maketitle

\begin{abstract}
Two popular and often applied methods to obtain two-dimensional point sets with the optimal order of $L_p$ discrepancy
are digit scrambling and symmetrization. In this paper we combine these two techniques and symmetrize
$b$-adic Hammersley point sets scrambled with arbitrary permutations. It is already known that these modifications indeed assure that the $L_p$ discrepancy is of optimal order $\mathcal{O}\left(\sqrt{\log{N}}/N\right)$ for $p\in [1,\infty)$ in contrast to the classical Hammersley point set. We prove an exact formula for the $L_2$ discrepancy of these point sets for special permutations. 
We also present the permutations which lead to the lowest $L_2$ discrepancy for every base $b\in\{2,\dots,27\}$ by employing computer search algorithms. 
\end{abstract}

\centerline{\begin{minipage}[hc]{130mm}{
}
\end{minipage}}

 \allowdisplaybreaks
 
\section{Introduction and statement of the result}

For a point set $\mathcal{P}=\{\bsx_0,\dots,\bsx_{N-1}\}$ with $N$ elements in the unit square $\left[0,1\right)^2$ its local discrepancy $E(x,y,\mathcal{P})$ is defined as
\[ E(x,y,\mathcal{P})=A(\left[0,x\right)\times \left[0,y\right),\mathcal{P})-Nxy \]
for $x, y \in \left(0,1\right]$. In this definition $A(\left[0,x \right)\times \left[0,y\right),\mathcal{P})$ is the number of indices $0\leq n\leq N-1$ satisfying $\bsx_n \in \left[0,x\right)\times \left[0,y\right)$.  Then the $L_p$ discrepancy of a point set $\mathcal{P}$ in $\left[0,1\right)^2$ is defined as
the $L_p$ norm of its local discrepancy divided by $N$, i.e.
\[L_p(\mathcal{P})=\frac{1}{N}\left(\int_{0}^{1}\int_{0}^{1} |E(x,y,\mathcal{P})|^p \rd x \rd y\right)^{\frac{1}{p}} \]
for $p\in[1,\infty)$. The $L_p$ discrepancy of point sets is related to the worst-case integration error of a quasi-Monte Carlo rule, see e.g. \cite{DP10,LP14,Nied92}. A well known result on the $L_p$ discrepancy is the following: for every $p\in [1,\infty)$ there exists a constant $c_p > 0$ with the 
property that for any point set $\mathcal{P}$ consisting of $N$ points in $[0, 1)^2$ we
have 
\begin{equation} \label{roth} L_p(\mathcal{P}) \geq c_p \frac{\sqrt{\log{N}}}{N}. \end{equation}
In this expression and throughout the paper, $\log$ denotes the natural logarithm.
This inequality was first shown by Roth \cite{Roth} for $p = 2$ and hence for all $p \in [2,\infty)$ and later by
Schmidt \cite{schX} for all $p\in(1,2)$. The end-point case $p=1$ was added by Hal\'{a}sz \cite{hala}.
We mention some more detailled results on the $L_2$ discrepancy. In \cite{FPPS09} it has been shown
that $$ \liminf_{N\to\infty}\inf_{\#\cP=N}\frac{NL_2(\mathcal{P})}{\sqrt{\log{N}}}\leq 0.179069\dots, $$ 
where the infimum is extended over all point sets $\cP$ with $N$ elements. This bound was obtained from digit scrambled Hammersley point sets as introduced in Definition~\ref{def1}.
Numerical results \cite{bil} suggest that symmetrized Fibonacci lattices yield a slightly better result,
such that the limes inferior could be bounded from above by $0.176006\dots$ However, this has not been strictly proven yet.
The best known lower bounds on the $L_2$ discrepancy were recently given in \cite{HL15}. We have
$$ \inf_{N\geq 2}\inf_{\#\cP=N}\frac{NL_2(\mathcal{P})}{\sqrt{\log{N}}}\geq 0.0515599\dots $$
and
$$ \limsup_{N\to\infty}\inf_{\#\cP=N}\frac{NL_2(\mathcal{P})}{\sqrt{\log{N}}}\geq 0.0610739\dots $$ 
In this paper, we consider generalized Hammersley point sets scrambled with arbitrary permutations and a symmetrized version thereof.

\begin{definition} \rm \label{def1} Let $b\geq 2$ and $n \geq 1$ be integers. 
Let $\mathfrak{S}_b$ be the set of all permutations of $\{0,1,\dots,b-1\}$ and $\Sigma=(\sigma_i)_{i=0}^{n-1} \in \mathfrak{S}_b^n$. We define the digit scrambled Hammersley point set associated to $\Sigma$ consisting of $N=b^n$ elements by
$$ \cR_{b,n}^{\Sigma}:=\left\{ \left(\sum_{i=0}^{n-1}\frac{\sigma_{n-1-i}(a_{n-1-i})}{b^{i+1}},\sum_{i=0}^{n-1}\frac{a_i}{b^{i+1}}\right): a_0,\dots,a_{n-1} \in \{0,1,\dots,b-1\}\right\}. $$
The choice $\Sigma=(\id)_{i=0}^{n-1}$, where $\id$ is the identity, yields the classical Hammersley point set in base $b$. Let $\tau \in \mathfrak{S}_b$ be given by $\tau(k)=b-1-k$ for $k\in \{0,1,\dots,b-1\}$.
In this paper we assume that for a fixed $\sigma \in \mathfrak{S}_b$ we have either $\sigma_i=\sigma$ or $\sigma_i=\tau \circ \sigma =: \overline{\sigma}$ for all $i \in \{0,\dots,n-1\}$, i.e. $\Sigma \in \{\sigma,\overline{\sigma}\}^n$. We define the number \begin{equation} \label{ln} l=l(\Sigma):=|\{i\in \{0,\dots,n-1\}: \sigma_i=\sigma\}|, \end{equation}
i. e. the number of components $\sigma_i$ of $\Sigma$ which equal $\sigma$. \\
Let $\sigma \in \mathfrak{S}_b$ and $\Sigma=(\sigma_i)_{i=0}^{n-1} \in \{\sigma,\overline{\sigma}\}^n$ be fixed. We put $\Sigma^{\ast}=(\sigma_i^{\ast})_{i=0}^{n-1} \in \{\sigma,\overline{\sigma}\}^n$, where $\sigma_i^{\ast}=\tau\circ\sigma_i$ for all $i \in \{0,\dots,n-1\}$. The symmetrized Hammersley point set (associated to $\Sigma$) consisting of $2b^n$ elements is then defined as 
   \begin{equation*}  \cR_{b,n}^{\Sigma, \mathrm{sym}}=\cR_{b,n}^{\Sigma}\cup \cR_{b,n}^{\Sigma^{\ast}}.\end{equation*}
We speak of a symmetrized point set, because $\cR_{b,n}^{\Sigma, \mathrm{sym}}$ can also be written as 
\begin{equation*} \cR_{b,n}^{\Sigma, \mathrm{sym}}=\cR_{b,n}^{\Sigma}\cup \left\{ \left(1-\frac{1}{b^n}-x,y\right) \, : \, (x,y) \in \cR_{b,n}^{\Sigma} \right\}. \end{equation*}
We also introduce the set $\mathcal{A}_b(\tau)=\{\sigma\in\mathfrak{S}_b:\, \sigma \circ \tau=\tau \circ \sigma\}$.
\end{definition}

We briefly survey several previous results on these point sets. 
The process of symmetrization and digit scrambling of point sets has been applied in discrepancy theory many times before. This is due to the fact
that the classical Hammersley point set fails to have optimal $L_p$-discrepancy for all $p\in[1,\infty)$, see e.g. \cite{FP}. The first two-dimensional point set with the optimal order of $L_2$-discrepancy was indeed found within symmetrized point sets by Davenport \cite{daven} in 1956. Halton and Zaremba \cite{HZ} introduced digit scrambling for the dyadic Hammersley point set in 1969 and showed that the modified point sets overcome the defect of the classical Hammersley point set and achieve an optimal $L_2$-discrepancy in the sense of \eqref{roth}. \\
The digit scrambled Hammersley point sets were studied further in several papers. We mention \cite{FP},
where only the case $\Sigma\in\{\id,\tau\}^n$ was considered. The results in this paper show that the $L_p$ discrepancy of the classical Hammersley point set is only of order $\mathcal{O}((\log{N})/N)$  for all $p\in[1,\infty)$ in all bases $b\geq 2$. However, the authors could also prove the existence of a $\Sigma\in\{\id,\tau\}^n$ such that $L_p(\cR_{b,n}^{\Sigma})=\mathcal{O}(\sqrt{\log{N}}/N)$ for all even positive integers $p$ and found an exact formula for the $L_2$ discrepancy of $\cR_{b,n}^{\Sigma}$ for an arbitrary $\Sigma\in\{\id,\tau\}^n$. The $L_2$ discrepancy was also studied in the general
setting $\Sigma\in\{\sigma,\overline{\sigma}\}^n$ for an arbitrary $\sigma\in\mathfrak{S}_b$ in \cite{FPPS09}. The authors obtained
$$ (b^n\, L_2(\cR_{b,n}^{\Sigma}))^2=\big(\Phi_b^{\sigma}\big)^2((n-2l)^2-n)+\mathcal{O}(n), $$
where $l$ is as in \eqref{ln} and where 
\begin{equation} \label{phicond} \Phi_b^{\sigma}=\frac{1}{b^2} \sum_{k=0}^{b-1}\sigma(k)k-\frac1b \left(\frac{b-1}{2}\right)^2.\end{equation} 
It follows from this formula that $ L_2(\cR_{b,n}^{\Sigma})=\mathcal{O}(\sqrt{\log{N}}/N) $
if and only if $|n-2l|=\mathcal{O}(\sqrt{n})$ or $\frac{1}{b} \sum_{k=0}^{b-1}\sigma(k)k=\left(\frac{b-1}{2}\right)^2$. In \cite{Kri1} it was shown that the same conditions are also sufficient and necessary for $L_p(\cR_{b,n}^{\Sigma})=\mathcal{O}(\sqrt{\log{N}}/N)$ for all $p\in [1,\infty)$. \\
There are also some known facts on the symmetrized point sets as introduced in Definition~\ref{def1}. The optimal order for the $L_2$ discrepancy of symmetrized generalized Hammersley point sets has already been obtained in \cite{Proi} as a corollary of results on the diaphony of generalized van der Corput sequences and later  in base $2$ as a special case of discrepancy estimates of so-called $(0,m,2)$-nets in \cite{LP01} with the aid of Walsh functions. In \cite{Kri1} it was proven that independently of $\Sigma\in\{\sigma,\overline{\sigma}\}^n$ the point set $\cR_{b,n}^{\Sigma,\sym}$ always achieves an $L_p$ discrepancy of order $\sqrt{\log{N}}/N$ for all bases $b\geq 2$ and for all $p\in[1,\infty).$ The proof requires tools from harmonic analysis, which have the drawback that they do not deliver exact formulas for the $L_p$ discrepancy. However,
an exact formula for the $L_2$ discrepancy of $\cR_{2,n}^{\Sigma,\sym}$ with an arbitrary $\Sigma\in\{\id,\tau\}^n$ was shown recently in \cite{Kri2}. We have
\begin{equation} \label{basis2} (2^{n+1}L_2(\cR_{2,n}^{\Sigma,\sym}))^2=\frac{n}{24}+\frac{11}{8}+\frac{1}{2^{n}}-\frac{1}{9\cdot 2^{2n+1}}. \end{equation}
This result demonstrates that in fact the $L_2$ discrepancy does not depend on $\Sigma$ at all, but only on the parameter $n$ which is connected
to the number of elements $N$ of $\cR_{2,n}^{\Sigma,\sym}$ via $N=2^{n+1}$.
The aim of this paper is to generalize \eqref{basis2} to arbitrary bases. We therefore prove the subsequent Theorem~\ref{theo}, which gives an exact formula for the $L_2$ discrepancy of $\cR_{b,n}^{\Sigma,\sym}$ for an arbitrary $\Sigma\in\{\sigma,\osigma\}^n$ with $\sigma\in\mathcal{A}_b(\tau)$. To this end, we need some notation that was initially introduced by Faure in \cite{Fau81}.

\begin{definition} \rm \label{deffaure} Let $\sigma\in \mathfrak{S}_b$ and let $\mathcal{Z}_b^{\sigma}=(\sigma(0)/b,\sigma(1)/b,\dots,\sigma(b-1)/b)$.
For $h\in\{0,1,\dots,b-1\}$ and $x\in[(k-1)/b,k/b)$, where $k\in \{1,\dots,b\}$, we define
\begin{equation*}
  \varphi_{b,h}^{\sigma}(x):= \begin{cases}
	                            A([0,h/b);k;\mathcal{Z}_b^{\sigma})-hx & \mbox{if } 0\leq h \leq \sigma(k-1), \\
															(b-h)x-A([h/b,1);k;\mathcal{Z}_b^{\sigma}) & \mbox{if } \sigma(k-1)< h <b .
	                        \end{cases}\end{equation*}
	In this definition, for a sequence $X=(x_M)_{M\geq 1}$, $A([x,y);N;X)$ denotes the number of indices $M$ with $1\leq M \leq N$ such that $x_M \in [x,y)$. The function $\varphi_{b,h}^{\sigma}$ is extended to the reals by periodicity, i.e. we have $\varphi_{b,h}^{\sigma}(x)=\varphi_{b,h}^{\sigma}(\{x\})$ for all $x\in\RR$. The notation $\{\cdot\}$ means the fractional part of an $x\in\RR$. We note that $\varphi_{b,0}^{\sigma}=0$ for any $\sigma$ and that $\varphi_{b,h}^{\sigma}(0)=0$ for any $\sigma$ and any $h$. We define some other functions which will appear in diverse parts of this paper.
	First, we put 
	$$ \varphi_b^{\sigma}:=\sum_{h=0}^{b-1}\varphi_{b,h}^{\sigma} \mbox{\quad and \quad} \varphi_b^{\sigma,(2)}:=\sum_{h=0}^{b-1}(\varphi_{b,h}^{\sigma})^2.$$
	We also set
	$$ \widetilde{\varphi}_b^{\sigma}:=\sum_{h=0}^{b-1}\varphi_{b,h}^{\sigma}\varphi_{b,h}^{\overline{\sigma}}, \quad \widetilde{\varphi}_{b,1}^{\sigma}:=\sum_{h=0}^{b-2}\varphi_{b,h+1}^{\sigma}\varphi_{b,h}^{\overline{\sigma}} \mbox{\quad and \quad} \widetilde{\varphi}_{b,2}^{\sigma}:=\sum_{h=0}^{b-2}\varphi_{b,h}^{\sigma}\varphi_{b,h+1}^{\overline{\sigma}}. $$
Finally, we define $\Phi_b^{\sigma}:=\frac1b \int_0^1 \varphi_b^{\sigma}(x)\rd x$ and analogously for $\Phi_b^{\sigma,(2)}$, $\widetilde{\Phi}_b^{\sigma}$, $\widetilde{\Phi}_{b,1}^{\sigma}$ and $\widetilde{\Phi}_{b,2}^{\sigma}$. The number $\Phi_b^{\sigma}$ is the same as given in \eqref{phicond} as it was shown in \cite[Lemma 5]{FPPS09}. 
\end{definition}

Now we are ready to state the main result of this paper.
\begin{theorem} \label{theo} Let $\sigma\in\mathcal{A}_b(\tau)$, $n\in\NN$ and $\Sigma\in\{\sigma,\overline{\sigma}\}^n$. Then we have
 \begin{align*}
    \left(2b^n\, L_2(\cR_{b,n}^{\Sigma,\sym})\right)^2=&\,nc_b^{\sigma}+\frac{11}{8}+\frac{1}{b^n}+\frac{1-9\cdot (-1)^b}{144b^{2n}},
 \end{align*}
 where $c_b^{\sigma}:=2\Phi_b^{\sigma,(2)}+\widetilde{\Phi}_b^{\sigma}+\frac12 \widetilde{\Phi}_{b,1}^{\sigma}+\frac12 \widetilde{\Phi}_{b,2}^{\sigma}$.
\end{theorem}

\begin{remark} \rm \label{arrang} Theorem~\ref{theo} demonstrates that $L_2(\cR_{b,n}^{\Sigma,\sym})$ does not depend on the distribution of $\sigma$ and $\osigma$ in $\Sigma\in\{\sigma,\overline{\sigma}\}^n$ at all, but only on the base $b$, on the permutation $\sigma\in\mathcal{A}_b(\tau)$ we choose and on the number of elements $N=2b^n$. Hence, for a fixed $\sigma\in\mathcal{A}_b(\tau)$
one should always choose $\Sigma=(\sigma,\sigma,\dots,\sigma)$ and $\Sigma^{\ast}=(\overline{\sigma},\overline{\sigma},\dots,\overline{\sigma})$ if one is only interested in a low $L_2$ discrepancy of $\cR_{b,n}^{\Sigma,\sym}$.
\end{remark}

\begin{example} \rm \label{example} We would like to derive results for the simplest case $\sigma=\id$. Let $\Sigma\in\{\id,\tau\}^n$ for some $n\in\NN$. \begin{comment}
We have
\begin{align*}
   \Phi_b^{\id,(2)}&=\frac{1}{90}b^2-\frac{1}{90b^2}, \\
   \widetilde{\Phi}_b^{\id}&=-\frac{7}{720}b^2-\frac{d_1}{720b^2} \text{\quad and} \\
	 \widetilde{\Phi}_{b,1}^{\id}&=\widetilde{\Phi}_{b,2}^{\id}=-\frac{7}{720}b^2-\frac{d_2}{720b^2}+\frac{1}{24},
\end{align*}
where $d_1=d_2=8$ for even bases $b$ and $d_1=-7$ and $d_2=23$
for odd bases $b$. This may be verified in the same manner as done in the proof of Lemma~\ref{id}. Inserting these results in Theorem~\ref{theo}, we obtain
for all $\Sigma\in\{\id,\tau\}^n$ \end{comment}
From Lemma~\ref{altern} in Section~\ref{numerical} we derive the formula
$$\left(2b^n\, L_{2}(\cR_{b,n}^{\Sigma,\sym})\right)^2=n\left(\frac{b^2}{360}+\frac{1}{24}-\frac{2}{45b^2}\right)+\frac{11}{8}+\frac{1}{b^n}+\frac{1-9\cdot(-1)^b}{144b^{2n}}.$$
 We remark that for $b=2$ this formula recovers \eqref{basis2}.
From \cite[Corollary 4]{FP} we have
$$ \min_{\Sigma\in\{\id,\tau\}^n}\left(b^n\,L_2(\cR_{b,n}^{\Sigma})\right)^2=n\left(\frac{b^2}{240}+\frac{1}{72}-\frac{13}{720b^2}\right)+\mathcal{O}(1). $$
This means that in the case $\Sigma\in\{\id,\tau\}^n$, symmetrizing yields asymptotically a lower $L_2$ discrepancy than digit scrambling for $b\geq 5$.
\end{example}
The structure of this paper is as follows: Section~\ref{proof} is divided into two subsections: In the first one we present the basic ideas for the proof of Theorem~\ref{theo} and defer some technical auxiliary results to the second subsection. In Section~\ref{numerical}, we present the numerical results and outline the methods and algorithms we used
to obtain them. In the final Section~\ref{conc} we point out the essential conclusions from our results. We would like to mention that our proof relies strongly on methods developed and used in the papers \cite{Fau81, FP, FPPS09}, amongst others.

\section{Proof of Theorem~\ref{theo}} \label{proof}

The formalism we use to verify Theorem~\ref{theo} is rather complicated and leads to several technical proofs. We therefore
would like to proceed in the following way: In the subsequent subsection, we present the high level structure of the proof, where
we try to avoid as many technicalities as possible. This subsection gives the reader the basic idea of the proof. We refer those who would like to fully understand all the details to Subsection~\ref{detail}.

\subsection{The basic steps of the proof}

The basic ingredient of our proof is an exact formula for the local discrepancy of $\cR_{b,n}^{\Sigma}$ (\cite[Lemma 1]{FP}), which goes back to H. Faure. 

\begin{lemma} \label{allgemein} For integers $1\leq \lambda,N \leq b^n$ we have
   $$ E\left(\frac{\lambda}{b^n},\frac{N}{b^n},\cR_{b,n}^{\Sigma}\right)=\sum_{j=1}^{n}\varphi_{b,\varepsilon_j(\lambda,N,\Sigma)}^{\sigma_{j-1}}\left(\frac{N}{b^j}\right). $$
	The numbers $\varepsilon_j(\lambda,N,\Sigma)$ for $j\in \{1,2,\dots,n\}$ are given as follows: For $1\leq \lambda <b^n$ with $b$-adic expansion
	$\lambda=\lambda_1b^{n-1}+\lambda_2b^{n-2}+\dots+\lambda_{n-1}b+\lambda_n$, we define
	$$ \Lambda_{j-1}=\Lambda_{j-1}(\lambda)=\lambda_jb^{n-j}+\dots\lambda_n. $$
	Then, for $1\leq N< b^n$ with $b$-adic expansion $N=N_{n-1}b^{n-1}+\dots+N_0$, we define
	$$ \nu_j=\nu_j(N,\Sigma)=\sigma_j(N_j)b^{n-j-1}+\dots+\sigma_{n-2}(N_{n-2})b+\sigma_{n-1}(N_{n-1}). $$
	Now we set $\varepsilon_n=\lambda_n$ and for fixed $1\leq j \leq n-1$ we set
	\[ \varepsilon_j=\varepsilon_j(\lambda,N,\Sigma)=\begin{cases}
	                                    0 & \mbox{if } 0\leq \Lambda_{j-1}\leq \nu_j, \\
																			h & \mbox{if } \nu_j+(h-1)b^{n-j}< \Lambda_{j-1}\leq \nu_j+hb^{n-j} \mbox { for } 1\leq h<b, \\
																			0 & \mbox{if } \nu_j+(b-1)b^{n-j}< \Lambda_{j-1}< b^{n-j+1}.
	                                 \end{cases}
	\]
For $\lambda=b^n$ or $N=b^n$ we set $\varepsilon_j(\lambda,N,\Sigma)=0$ for all $1\leq j \leq n$.
\end{lemma}

\begin{remark} \label{rem1} \rm Since the components of all points in $\cR_{b,n}^{\Sigma}$ are of the form $m/b^n$ for some $m\in\{0,\dots,b^n-1\}$, we have 
$$ E(x,y,\cR_{b,n}^{\Sigma})=E(x(n),y(n),\cR_{b,n}^{\Sigma})+b^n(x(n)y(n)-xy)   $$
for all $x,y\in (0,1]$, where we set $x(n):=\min\{m/b^n\geq x:\, m\in\{1,\dots,b^n\}\}$ for an $x\in(0,1]$. This relation has already been remarked in
\cite[Remark 3]{FP}.
\end{remark}

Throughout this paper, we write $E_1(x,y)$ for the local discrepancy of $\cR_{b,n}^{\Sigma}$, $E_2(x,y)$ for the local discrepancy of $\cR_{b,n}^{\Sigma^{\ast}}$ and $E_{\sym}(x,y)$ for the local discrepancy of $\cR_{b,n}^{\Sigma,\sym}$. The next lemma provides a formula for $E_{\sym}(x,y)$. This formula follows immediately from the fact that $\cR_{b,n}^{\Sigma,\sym}$ is defined as the union of $\cR_{b,n}^{\Sigma}$ and $\cR_{b,n}^{\Sigma^{\ast}}$. For the simple, elementary proof we refer to \cite[Lemma 2]{Kri2}.

\begin{lemma} \label{symrelation} We have $E_{\sym}(x,y)=E_1(x,y)+E_2(x,y)$ for all $x,y\in (0,1]$.
\end{lemma}
Now we can start with the proof. 
 With the definition of the $L_2$ discrepancy and with Lemma~\ref{symrelation} we obtain
\begin{align} (2b^nL_2(\cR_{b,n}^{\Sigma,\sym}))^2=& \int_{0}^{1} \int_{0}^{1}  (E_{\sym}(x,y))^2 \rd x \rd y \nonumber \\ \nonumber
 =&\int_{0}^{1} \int_{0}^{1}  (E_1(x,y))^2 \rd x \rd y+\int_{0}^{1} \int_{0}^{1}  (E_2(x,y))^2 \rd x \rd y\\ \nonumber
 &+2\int_{0}^{1} \int_{0}^{1}  E_1(x,y)E_2(x,y) \rd x\rd y\\ \nonumber
 =&(b^n L_2(\cR_{b,n}^{\Sigma}))^2+(b^n L_2(\cR_{b,n}^{\Sigma^{\ast}}))^2 \\ \label{letztesintegral}
&+2\int_{0}^{1} \int_{0}^{1}  E_1(x,y)E_2(x,y) \rd x \rd y.
\end{align}
At this point, we make use of a previous result on the $L_2$ discrepancy of $\cR_{b,n}^{\Sigma}$.
In case that $\sigma\in \mathcal{A}_b(\tau)$ the authors of \cite{FPPS09} could show a completely exact formula for this quantity. We recall the
definition $l=l(\Sigma)=|\{i\in \{0,\dots,n-1\}: \sigma_i=\sigma\}|$ as given in \eqref{ln}. The following result is \cite[Theorem 2]{FPPS09}.

\begin{lemma}[Faure et. al.] \label{exact} Let $\sigma\in \mathcal{A}_b(\tau)$ and $\Sigma\in\{\sigma,\osigma\}^n$. Then we have
\begin{align*}
    (b^n\, L_2(\cR_{b,n}^{\Sigma}))^2=&\big(\Phi_b^{\sigma}\big)^2((n-2l)^2-n)+\Phi_b^{\sigma}\left(1-\frac{1}{2b^n}\right)(2l-n) \\
              &+ n\Phi_b^{\sigma,(2)}+\frac38+\frac{1}{4b^n}-\frac{1}{72b^{2n}}.
\end{align*}
\end{lemma}

Lemma~\ref{exact} yields 
\begin{align*} (b^n& L_2(\cR_{b,n}^{\Sigma}))^2+(b^n L_2(\cR_{b,n}^{\Sigma^{\ast}}))^2 \\
  &=\big(\Phi_b^{\sigma}\big)^2(2n^2+8l^2-2n-8ln)+2n\Phi_b^{\sigma,(2)}+\frac34+\frac{1}{2b^n}-\frac{1}{36b^{2n}}.
\end{align*}
Here we regarded the obvious fact that $\Sigma^{\ast}$ contains $n-l$ entries equal to $\id$ whenever $\Sigma$ contains $l$ of such entries.
We examine \eqref{letztesintegral} and therefore regard Remark~\ref{rem1} to write
\begin{align*}
    \int_{0}^{1}& \int_{0}^{1}  E_1(x,y)E_2(x,y) \rd x \rd y \\
       =& \int_0^1 \int_0^1 \left(E_1(x(n),y(n))+b^n(x(n)y(n)-xy)\right) \\
       &\times \left(E_2(x(n),y(n))+b^n(x(n)y(n)-xy)\right)\rd x \rd y \\
       =& \int_0^1 \int_0^1 E_1(x(n),y(n))E_2(x(n),y(n))\rd x \rd y \\
       &+ b^n\int_0^1 \int_0^1 E_1(x(n),y(n))(x(n)y(n)-xy)\rd x \rd y \\
       &+ b^n\int_0^1 \int_0^1 E_2(x(n),y(n))(x(n)y(n)-xy)\rd x \rd y \\
       &+ b^{2n} \int_0^1 \int_0^1 (x(n)y(n)-xy)^2 \rd x \rd y=:\Sigma_1+\Sigma_2+\Sigma_3+\Sigma_4.
\end{align*}
From the proof of \cite[Theorem 2]{FPPS09} we already know that
$$ \Sigma_4=\frac{25}{72}+\frac{1}{4b^n}+\frac{1}{72b^{2n}} $$
and
$$ \Sigma_2=(2l-n) \left(\frac12-\frac{1}{4b^n}\right)\Phi_b^{\sigma}. $$
By replacing $l$ by $n-l$ in the result for $\Sigma_2$ we obtain
$$ \Sigma_3=(2(n-l)-n) \left(\frac12-\frac{1}{4b^n}\right)\Phi_b^{\sigma} $$
and therefore $\Sigma_2+\Sigma_3=0$.
It remains to evaluate $\Sigma_1$. In the following, we do nothing else but inserting Lemma~\ref{allgemein} for $E_1(x,y)$ and $E_2(x,y)$, and then separating those indices $i\in\{1,\dots,n\}$ where $\sigma_{i-1}=\sigma$ from those where $\sigma_{i-1}=\overline{\sigma}$. We have
\begin{align*}
    \Sigma_1=& \frac{1}{b^{2n}}\sum_{\lambda,N=1}^{b^n}E_1\left(\frac{\lambda}{b^n},\frac{N}{b^n}\right)E_2\left(\frac{\lambda}{b^n},\frac{N}{b^n}\right) \\
       =& \frac{1}{b^{2n}}\sum_{\lambda,N=1}^{b^n}\left(\sum_{i=1}^{n}\varphi_{b,\varepsilon_i(\lambda,N,\Sigma)}^{\sigma_{i-1}}\left(\frac{N}{b^i}\right)\right)
            \left(\sum_{j=1}^{n}\varphi_{b,\varepsilon_j(\lambda,N,\Sigma^{\ast})}^{\sigma_{j-1}^{\ast}}\left(\frac{N}{b^j}\right)\right)\\
       =& \frac{1}{b^{2n}}\sum_{\lambda,N=1}^{b^n}\left(\sum_{\substack{i=1\\ \sigma_{i-1}=\sigma}}^{n}\varphi_{b,\varepsilon_i(\lambda,N,\Sigma)}^{\sigma}\left(\frac{N}{b^i}\right)+\sum_{\substack{i=1\\ \sigma_{i-1}=\overline{\sigma}}}^{n}\varphi_{b,\varepsilon_i(\lambda,N,\Sigma)}^{\overline{\sigma}}\left(\frac{N}{b^i}\right)\right)  \\
       &\times \left(\sum_{\substack{j=1\\ \sigma_{j-1}=\sigma}}^{n}\varphi_{b,\varepsilon_j(\lambda,N,\Sigma^{\ast})}^{\overline{\sigma}}\left(\frac{N}{b^j}\right)+\sum_{\substack{j=1\\ \sigma_{j-1}=\overline{\sigma}}}^{n}\varphi_{b,\varepsilon_j(\lambda,N,\Sigma^{\ast})}^{\sigma}\left(\frac{N}{b^j}\right)\right)\\ 
    =& \frac{1}{b^{2n}}\sum_{\lambda,N=1}^{b^n}\left(\sum_{\substack{i=1\\ \sigma_{i-1}=\sigma}}^{n}\varphi_{b,\varepsilon_i(\lambda,N,\Sigma)}^{\sigma}\left(\frac{N}{b^i}\right)\right)\left(\sum_{\substack{j=1\\ \sigma_{j-1}=\sigma}}^{n}\varphi_{b,\varepsilon_j(\lambda,N,\Sigma^{\ast})}^{\overline{\sigma}}\left(\frac{N}{b^j}\right)\right) \\   
    &+\frac{1}{b^{2n}}\sum_{\lambda,N=1}^{b^n}\left(\sum_{\substack{i=1\\ \sigma_{i-1}=\sigma}}^{n}\varphi_{b,\varepsilon_i(\lambda,N,\Sigma)}^{\sigma}\left(\frac{N}{b^i}\right)\right)\left(\sum_{\substack{j=1\\ \sigma_{j-1}=\overline{\sigma}}}^{n}\varphi_{b,\varepsilon_j(\lambda,N,\Sigma^{\ast})}^{\sigma}\left(\frac{N}{b^j}\right)\right) \\ 
    &+\frac{1}{b^{2n}}\sum_{\lambda,N=1}^{b^n}\left(\sum_{\substack{i=1\\ \sigma_{i-1}=\overline{\sigma}}}^{n}\varphi_{b,\varepsilon_i(\lambda,N,\Sigma)}^{\overline{\sigma}}\left(\frac{N}{b^i}\right)\right)\left(\sum_{\substack{j=1\\ \sigma_{j-1}=\sigma}}^{n}\varphi_{b,\varepsilon_j(\lambda,N,\Sigma^{\ast})}^{\overline{\sigma}}\left(\frac{N}{b^j}\right)\right) \\ 
    &+\frac{1}{b^{2n}}\sum_{\lambda,N=1}^{b^n}\left(\sum_{\substack{i=1\\ \sigma_{i-1}=\overline{\sigma}}}^{n}\varphi_{b,\varepsilon_i(\lambda,N,\Sigma)}^{\overline{\sigma}}\left(\frac{N}{b^i}\right)\right)\left(\sum_{\substack{j=1\\ \sigma_{j-1}=\overline{\sigma}}}^{n}\varphi_{b,\varepsilon_j(\lambda,N,\Sigma^{\ast})}^{\sigma}\left(\frac{N}{b^j}\right)\right) \\
    =&:S_1+S_2+S_3+S_4.
\end{align*}
Now, for the first time in this proof, we have to deal with the functions $\varphi_{b,h}^{\sigma}$ which appear in Lemma~\ref{allgemein}. First, we only need results that have already been proven in previous papers.
The proofs of the following auxiliary results can be found in \cite[Lemma 2]{FP}, \cite[Lemma 3]{FPPS09} and \cite[Lemma 4]{FPPS09}, respectively.

\begin{lemma} \label{previous} Let $1\leq N \leq b^n$, $1\leq i < j \leq n$ and $\sigma\in\mathfrak{S}_b$. Then we have for $\sigma_i,\sigma_j \in \{\sigma,\overline{\sigma}\}$
    $$ \sum_{\lambda=1}^{b^n}\big(\varphi_{b,\varepsilon_i(\lambda,N,\Sigma_1)}^{\sigma_i}\left(\frac{N}{b^i}\right)\big)
		                       \big(\varphi_{b,\varepsilon_j(\lambda,N,\Sigma_2)}^{\sigma_j}\left(\frac{N}{b^j}\right)\big)
													=b^{n-2}\varphi_b^{\sigma_i}\left(\frac{N}{b^i}\right)\varphi_b^{\sigma_j}\left(\frac{N}{b^j}\right),$$
 where $\Sigma_1,\Sigma_2\in\{\sigma,\overline{\sigma}\}^n$ may be different.
Then, for $1\leq i < j \leq n$ and an arbitrary permutation $\sigma\in\mathfrak{S}_b$, we have 
$$ \sum_{N=1}^{b^n}\varphi_b^{\sigma}\left(\frac{N}{b^i}\right)\varphi_b^{\sigma}\left(\frac{N}{b^j}\right)=b^{n+2}\big(\Phi_b^{\sigma}\big)^2. $$
Finally, for $\sigma\in\mathfrak{S}_b$, $\overline{\sigma}=\tau \circ \sigma$ and any $h\in\{0,\dots,b-1\}$, we also have the relations
$$ \varphi_{b,h}^{\overline{\sigma}}=-\varphi_{b,b-h}^{\sigma} $$
 and, as a result from that, $\varphi_b^{\sigma}=-\varphi_b^{\overline{\sigma}}$ and $\varphi_b^{\sigma,(2)}=\varphi_b^{\overline{\sigma},(2)}$.
\end{lemma}

 We change the summation order and use the statements of Lemma~\ref{previous} to compute 
\begin{align*}
    S_2=&\frac{1}{b^{2n}}\sum_{\substack{i,j=1 \\ \sigma_{i-1}=\sigma \\ \sigma_{j-1}=\overline{\sigma}}}^{n}\sum_{N=1}^{b^n}b^{n-2}\varphi_b^{\sigma}\left(\frac{N}{b^i}\right)\varphi_b^{\sigma}\left(\frac{N}{b^j}\right)
    =\frac{1}{b^{2n}}b^{n-2}\sum_{\substack{i,j=1 \\ \sigma_{i-1}=\sigma \\ \sigma_{j-1}=\overline{\sigma}}}^{n}b^{n+2}\big(\Phi_b^{\sigma}\big)^2 
    =l(n-l)\big(\Phi_b^{\sigma}\big)^2.
\end{align*}
Similarly, we show $S_3=l(n-l)\big(\Phi_b^{\sigma}\big)^2=S_2$. To evaluate $S_1$, we have to distinguish the cases where $i\neq j$ and where $i=j$.
The first case can be treated analogously to $S_2$ and $S_3$.
Hence,
\begin{align*}
   S_1=&\frac{1}{b^{2n}}\sum_{\substack{i,j=1 \\ \sigma_{i-1}=\sigma \\ \sigma_{j-1}=\sigma \\ i\neq j}}^{n}\sum_{N=1}^{b^n}b^{n-2}\varphi_b^{\sigma}\left(\frac{N}{b^i}\right)\varphi_b^{\overline{\sigma}}\left(\frac{N}{b^j}\right) \\
   &+\frac{1}{b^{2n}}\sum_{\substack{j=1 \\ \sigma_{j-1}=\sigma }}^{n}\sum_{\lambda,N=1}^{b^n}\varphi_{b,\varepsilon_j(\lambda,N,\Sigma)}^{\sigma}\left(\frac{N}{b^j}\right)\varphi_{b,\varepsilon_j(\lambda,N,\Sigma^{\ast})}^{\overline{\sigma}}\left(\frac{N}{b^j}\right) \\
   =& -l(l-1)\big(\Phi_b^{\sigma}\big)^2+\frac{1}{b^{2n}}\sum_{\substack{j=1 \\ \sigma_{j-1}=\sigma }}^{n}\sum_{\lambda,N=1}^{b^n}\varphi_{b,\varepsilon_j(\lambda,N,\Sigma)}^{\sigma}\left(\frac{N}{b^j}\right)\varphi_{b,\varepsilon_j(\lambda,N,\Sigma^{\ast})}^{\overline{\sigma}}\left(\frac{N}{b^j}\right).
\end{align*}
In the same way we show
$$ S_4=-(n-l)(n-l-1)\big(\Phi_b^{\sigma}\big)^2+\frac{1}{b^{2n}}\sum_{\substack{j=1 \\ \sigma_{j-1}=\overline{\sigma} }}^{n}\sum_{\lambda,N=1}^{b^n}\varphi_{b,\varepsilon_j(\lambda,N,\Sigma)}^{\osigma}\left(\frac{N}{b^j}\right)\varphi_{b,\varepsilon_j(\lambda,N,\Sigma^{\ast})}^{\sigma}\left(\frac{N}{b^j}\right). $$
From the proof of Lemma~\ref{main} we observe that
$$ \sum_{\lambda,N=1}^{b^n}\varphi_{b,\varepsilon_j(\lambda,N,\Sigma)}^{\osigma}\left(\frac{N}{b^j}\right)\varphi_{b,\varepsilon_j(\lambda,N,\Sigma^{\ast})}^{\sigma}\left(\frac{N}{b^j}\right)=\sum_{\lambda,N=1}^{b^n}\varphi_{b,\varepsilon_j(\lambda,N,\Sigma)}^{\sigma}\left(\frac{N}{b^j}\right)\varphi_{b,\varepsilon_j(\lambda,N,\Sigma^{\ast})}^{\overline{\sigma}}\left(\frac{N}{b^j}\right). $$
Summarizing, we have
\begin{align*}
\Sigma_1=(-n^2-4l^2+n+4ln)\big(\Phi_b^{\sigma}\big)^2+\frac{1}{b^{2n}}\sum_{\substack{j=1 }}^{n}\sum_{\lambda,N=1}^{b^n}\varphi_{b,\varepsilon_j(\lambda,N,\Sigma)}^{\sigma}\left(\frac{N}{b^j}\right)\varphi_{b,\varepsilon_j(\lambda,N,\Sigma^{\ast})}^{\overline{\sigma}}\left(\frac{N}{b^j}\right)
\end{align*}
and thus, by putting all results together, we arrive at
\begin{align}\nonumber
 \left(2b^n\, L_2(\cR_{b,n}^{\Sigma,\sym})\right)^2=& 2n\Phi_b^{\sigma,(2)}+\frac{13}{9}+\frac{1}{b^n}\\ &+\frac{2}{b^{2n}}\sum_{\substack{j=1 }}^{n}\sum_{\lambda,N=1}^{b^n}\varphi_{b,\varepsilon_j(\lambda,N,\Sigma)}^{\sigma}\left(\frac{N}{b^j}\right)\varphi_{b,\varepsilon_j(\lambda,N,\Sigma^{\ast})}^{\overline{\sigma}}\left(\frac{N}{b^j}\right). \label{final}
\end{align}
We observe that the remaining step to finally prove Theorem~\ref{theo} is the evaluation of the expression 
$$ \frac{2}{b^{2n}}\sum_{\substack{j=1 }}^{n}\sum_{\lambda,N=1}^{b^n}\varphi_{b,\varepsilon_j(\lambda,N,\Sigma)}^{\sigma}\left(\frac{N}{b^j}\right)\varphi_{b,\varepsilon_j(\lambda,N,\Sigma^{\ast})}^{\overline{\sigma}}\left(\frac{N}{b^j}\right). $$
This is the most difficult and technical part of the proof, and all the lemmas we present in the following subsection aim at calculating this term. The final result is stated in Lemma~\ref{main}. Inserting the formula given in this lemma (and in Remark~\ref{simpleform}) into \eqref{final} completes the proof of Theorem~\ref{theo}.

\subsection{The details of the proof} \label{detail}

To guide the reader through the proofs in this subsection, we explain the basic ideas in a few lines preceeding the corresponding lemma, respectively.\\
Lemma~\ref{ugly} is the only lemma where we need the complicated definition of the numbers $\varepsilon_j(\lambda,N,\Sigma)$ appearing in Lemma~\ref{allgemein}.
The proof of this lemma may appear extremely technical on first look, but in fact we only apply basic combinatorial considerations. The main concern is
to investigate for which integers $\lambda\in\{1,\dots,b^n\}$ the numbers $\varepsilon_j(\lambda,N,\Sigma)$ and $\varepsilon_j(\lambda,N,\Sigma^{\ast})$ take certain values $h,h+1\in\{0,\dots,b-1\}$ simultaneously.
\begin{lemma} \label{ugly} Let $\sigma\in\mathfrak{S}_b$. For all $1\leq N \leq b^n$ and $1\leq j \leq n-1$ we have
 \begin{align*} \sum_{\lambda=1}^{b^n}&\varphi_{b,\varepsilon_j(\lambda,N,\Sigma)}^{\sigma}\left(\frac{N}{b^j}\right)
           \varphi_{b,\varepsilon_j(\lambda,N,\Sigma^{\ast})}^{\overline{\sigma}}\left(\frac{N}{b^j}\right) \\
				  	=&  \begin{cases}							b^{n-1}\widetilde{\varphi}_b^{\sigma}\left(\frac{N}{b^j}\right)+b^{j-1}(b^{n-j}-1-2\nu_j(N,\Sigma))\left(\widetilde{\varphi}_{b,1}^{\sigma}\left(\frac{N}{b^j}\right)-\widetilde{\varphi}_{b}^{\sigma}\left(\frac{N}{b^j}\right)\right)
						 & \\ \hspace{10 cm}\mbox{if } \nu_j(N,\Sigma)<\frac{b^{n-j}-1}{2}, \\					b^{n-1}\widetilde{\varphi}_b^{\sigma}\left(\frac{N}{b^j}\right)+b^{j-1}(2\nu_j(N,\Sigma)+1-b^{n-j})\left(\widetilde{\varphi}_{b,2}^{\sigma}\left(\frac{N}{b^j}\right)-\widetilde{\varphi}_{b}^{\sigma}\left(\frac{N}{b^j}\right)\right)
						 & \\ \hspace{10 cm}\mbox{if } \nu_j(N,\Sigma)\geq\frac{b^{n-j}-1}{2}.
								\end{cases} \end{align*}
If $j=n$, then we have
$$ \sum_{\lambda=1}^{b^n}\varphi_{b,\varepsilon_n(\lambda,N,\Sigma)}^{\sigma}\left(\frac{N}{b^n}\right)
           \varphi_{b,\varepsilon_n(\lambda,N,\Sigma^{\ast})}^{\overline{\sigma}}\left(\frac{N}{b^n}\right)
					=b^{n-1}\widetilde{\varphi}_b^{\sigma}\left(\frac{N}{b^n}\right).$$
\end{lemma}
\begin{proof} The case $N=b^n$ is trivial since then the left- and the right-hand-sides of the above equality are zero. We therefore assume $1\leq N <b^n$ now. We first show the case $j=n$. Since $\varepsilon_n(\lambda,N,\Sigma)=\varepsilon_n(\lambda,N,\Sigma^{\ast})=\lambda_n$
  by definition, we can write
	\begin{align*}
	   \sum_{\lambda=1}^{b^n}\varphi_{b,\varepsilon_n(\lambda,N,\Sigma)}^{\sigma}\left(\frac{N}{b^n}\right)
           \varphi_{b,\varepsilon_n(\lambda,N,\Sigma^{\ast})}^{\overline{\sigma}}\left(\frac{N}{b^n}\right)
					 =\sum_{h=0}^{b-1}\varphi_{b,h}^{\sigma}\left(\frac{N}{b^n}\right)\varphi_{b,h}^{\overline{\sigma}}\left(\frac{N}{b^n}\right)\sum_{\substack{\lambda=1 \\ \lambda_n=h}}^{b^n}1=b^{n-1}\widetilde{\varphi}_b^{\sigma}\left(\frac{N}{b^n}\right).
	\end{align*}

 We fix $j\in\{1,\dots,n-1\}$, $N\in\{1,\dots,b^n-1\}$ and $\Sigma\in\{\sigma,\overline{\sigma}\}^n$.
	We have to distinguish between two cases. Let us first assume that $\nu_j(N,\Sigma)<\nu_j(N,\Sigma^{\ast})$. Then we can either have 
	$$\varepsilon_j(\lambda,N,\Sigma)=\varepsilon_j(\lambda,N,\Sigma^{\ast}) \text{\, or \,} \varepsilon_j(\lambda,N,\Sigma)=\varepsilon_j(\lambda,N,\Sigma^{\ast})+1.$$
	We count the number of $\Lambda_{j-1}$ such that $\varepsilon_j(\lambda,N,\Sigma)=\varepsilon_j(\lambda,N,\Sigma^{\ast})=h$ for any $h\in\{0,\dots,b-1\}$.
	For $h\in\{1,\dots,b-1\}$ these $\Lambda_{j-1}$ are given by $\nu_j(N,\Sigma)+(h-1)b+z$ for $z\in\{\nu_j(N,\Sigma^{\ast})-\nu_j(N,\Sigma)+1,\dots, b^{n-j}\}$ and for $h=0$ the corresponding $\Lambda_{j-1}$ are $0,\dots, \nu_j(N,\Sigma)$ and 
	$\nu_j(N,\Sigma^{\ast})+(b-1)b^{n-j}+1,\dots,b^{n-j+1}-1$. Hence for all $h\in\{0,\dots,b-1\}$ we have $b^{n-j}-(\nu_j(N,\Sigma^{\ast})-\nu_j(N,\Sigma))$ values for $\Lambda_{j-1}$ such that $\varepsilon_j(\lambda,N,\Sigma)=\varepsilon_j(\lambda,N,\Sigma^{\ast})=h$. Since there are always
	$b^{j-1}$ elements $\lambda\in\{1,\dots,b^n\}$ with the same $\Lambda_{j-1}$ we have proven
	\begin{equation} \label{zwischen}\sum_{\substack{\lambda=1\\ \{\lambda:\, \varepsilon_j(\lambda,N,\Sigma)=\varepsilon_j(\lambda,N,\Sigma^{\ast})=h\}}}^{b^n}1=b^{j-1}(b^{n-j}-(\nu_j(N,\Sigma^{\ast})-\nu_j(N,\Sigma))).\end{equation}
	For $h\in\{0,\dots,b-2\}$, we have $\varepsilon_j(\lambda,N,\Sigma)=h+1=\varepsilon_j(\lambda,N,\Sigma^{\ast})+1$ for $\Lambda_{j-1}$ of the form $\nu_j(N,\Sigma)+hb+z$ for $z\in\{1,\dots,\nu_j(N,\Sigma^{\ast})-\nu_j(N,\Sigma)\}$. Hence we have
	\begin{equation} \label{zwischen2}\sum_{\substack{\lambda=1\\ \{\lambda:\, \varepsilon_j(\lambda,N,\Sigma)=h+1, \, \varepsilon_j(\lambda,N,\Sigma^{\ast})=h\}}}^{b^n}1=b^{j-1}(\nu_j(N,\Sigma^{\ast})-\nu_j(N,\Sigma))\end{equation}
	for all $h\in\{0,\dots,b-2\}$. Here we simply neglect the also possible case $\varepsilon_j(\lambda,N,\Sigma)=0,\varepsilon_j(\lambda,N,\Sigma^{\ast})=b-1$, since the corresponding summands in the sum
	$$ \sum_{\lambda=1}^{b^n}\varphi_{b,\varepsilon_j(\lambda,N,\Sigma)}^{\sigma}\left(\frac{N}{b^j}\right)
           \varphi_{b,\varepsilon_j(\lambda,N,\Sigma^{\ast})}^{\overline{\sigma}}\left(\frac{N}{b^j}\right)$$
	are zero anyway.
In the second case $\nu_j(N,\Sigma)\geq\nu_j(N,\Sigma^{\ast})$ we only have the possibilities 
$$\varepsilon_j(\lambda,N,\Sigma)=\varepsilon_j(\lambda,N,\Sigma^{\ast}) \text{\, or \,} \varepsilon_j(\lambda,N,\Sigma)+1=\varepsilon_j(\lambda,N,\Sigma^{\ast}).$$
 Apart from that, the situation is quite the same as in the first case and we have
	$$\sum_{\substack{\lambda=1\\ \{\lambda:\, \varepsilon_j(\lambda,N,\Sigma)=\varepsilon_j(\lambda,N,\Sigma^{\ast})=h\}}}^{b^n}1=b^{j-1}(b^{n-j}-(\nu_j(N,\Sigma)-\nu_j(N,\Sigma^{\ast})))$$
	for all $h\in\{0,\dots,b-1\}$ and
	$$\sum_{\substack{\lambda=1\\ \{\lambda:\, \varepsilon_j(\lambda,N,\Sigma)=h, \, \varepsilon_j(\lambda,N,\Sigma^{\ast})=h+1\}}}^{b^n}1=b^{j-1}(\nu_j(N,\Sigma)-\nu_j(N,\Sigma^{\ast}))$$ for all $h\in\{0,\dots,b-2\}$.
	 Next we prove the relation $\nu_j(N,\Sigma^{\ast})=b^{n-j}-1-\nu_j(N,\Sigma)$. From the definition of $\nu_j(N,\Sigma^{\ast})$   we find 
	\begin{align*}
        \nu_j(N,\Sigma^{\ast})=& \sigma^{\ast}_j(N_j)b^{n-j-1}+\dots+\sigma^{\ast}_{n-2}(N_{n-2})b+\sigma^{\ast}_{n-1}(N_{n-1}) \\
				=& (b-1-\sigma_j(N_j))b^{n-j-1}+\dots+(b-1-\sigma_{n-2}(N_{n-2}))b\\ 
				  &+(b-1-\sigma_{n-1}(N_{n-1})) \\
				=& (b-1)(b^{n-j-1}+\dots+b+1)-\nu_j(N,\Sigma)=b^{n-j}-1-\nu_j(N,\Sigma).
	\end{align*}
	This identity yields the equivalence of $\nu_j(N,\Sigma)<\nu_j(N,\Sigma^{\ast})$ and $\nu_j(N,\Sigma)<\frac{b^{n-j}-1}{2}$
	as well as the equivalence of $\nu_j(N,\Sigma)\geq\nu_j(N,\Sigma^{\ast})$ and $\nu_j(N,\Sigma)\geq\frac{b^{n-j}-1}{2}$.
	Now in the case $\nu_j(N,\Sigma)<\frac{b^{n-j}-1}{2}$ we find
	\begin{align*}
	    \sum_{\lambda=1}^{b^n}\varphi_{b,\varepsilon_j(\lambda,N,\Sigma)}^{\sigma}\left(\frac{N}{b^j}\right)&
           \varphi_{b,\varepsilon_j(\lambda,N,\Sigma^{\ast})}^{\overline{\sigma}}\left(\frac{N}{b^j}\right) \\=&
					 \sum_{h=0}^{b-1}\varphi_{b,h}^{\sigma}\left(\frac{N}{b^j}\right)\varphi_{b,h}^{\overline{\sigma}}\left(\frac{N}{b^j}\right)
					\sum_{\substack{\lambda=1\\ \{\lambda:\, \varepsilon_j(\lambda,N,\Sigma)=\varepsilon_j(\lambda,N,\Sigma^{\ast})=h\}}}^{b^n}1 \\
					&+ \sum_{h=0}^{b-2}\varphi_{b,h+1}^{\sigma}\left(\frac{N}{b^j}\right)\varphi_{b,h}^{\overline{\sigma}}\left(\frac{N}{b^j}\right)
					\sum_{\substack{\lambda=1\\ \{\lambda:\, \varepsilon_j(\lambda,N,\Sigma)=h+1,\varepsilon_j(\lambda,N,\Sigma^{\ast})=h\}}}^{b^n}1.
	\end{align*}
	Using \eqref{zwischen}, \eqref{zwischen2} and the definition of $\widetilde{\varphi}_b^{\sigma}$ and $\widetilde{\varphi}_{b,1}^{\sigma}$, this leads to
	\begin{align*}
	    \sum_{\lambda=1}^{b^n}\varphi_{b,\varepsilon_j(\lambda,N,\Sigma)}^{\sigma}\left(\frac{N}{b^j}\right)&
           \varphi_{b,\varepsilon_j(\lambda,N,\Sigma^{\ast})}^{\overline{\sigma}}\left(\frac{N}{b^j}\right) \\=&
					b^{j-1}(b^{n-j}-(\nu_j(N,\Sigma^{\ast})-\nu_j(N,\Sigma)))\widetilde{\varphi}_b^{\sigma}\left(\frac{N}{b^j}\right)\\
			&+b^{j-1}(\nu_j(N,\Sigma^{\ast})-\nu_j(N,\Sigma))\widetilde{\varphi}_{b,1}^{\sigma}\left(\frac{N}{b^j}\right) \\
			=&b^{n-1}\widetilde{\varphi}_b^{\sigma}\left(\frac{N}{b^j}\right)+
			b^{j-1}(\nu_j(N,\Sigma^{\ast})-\nu_j(N,\Sigma))\left(\widetilde{\varphi}_{b,1}^{\sigma}\left(\frac{N}{b^j}\right)-\widetilde{\varphi}_{b}^{\sigma}\left(\frac{N}{b^j}\right)\right).\end{align*}
			By applying the above relation between $\nu_j(N,\Sigma)$ and $\nu_j(N,\Sigma^{\ast})$ we find
			$$ \nu_j(N,\Sigma^{\ast})-\nu_j(N,\Sigma)=b^{n-j}-1-2\nu_j(N,\Sigma), $$
			which yields the claim of this lemma in the case $\nu_j(N,\Sigma)<\frac{b^{n-j}-1}{2}$. The other case can be completed
			analogously.
\end{proof}

We are now concerned with the task to compute sums of the form $\sum_{N=1}^{b^j}\widetilde{\varphi}_b^{\sigma}\left(\frac{N}{b^j}\right)$ (and analogous sums for $\widetilde{\varphi}_{b,1}^{\sigma}$ and $\widetilde{\varphi}_{b,2}^{\sigma}$). Let us first consider such sums for the special case $\sigma=\id$, since in this case the functions $\varphi_{b,h}^{\id}$, which we introduced in Definition~\ref{deffaure}, can be written down in a simple way. The rest of the proof of the subsequent Lemma~\ref{id} contains evaluations of elementary sums and integrals.

\begin{lemma} \label{id} For all $j\in\{1,\dots,n\}$ we have
   \begin{eqnarray*}\sum_{N=1}^{b^j}\widetilde{\varphi}_b^{\id}\left(\frac{N}{b^j}\right)&=&b^j\left( \int_0^1 \widetilde{\varphi}_b^{\id}(x)\rd x+\frac{A_b(j,\id)}{2}\right), \\
	              \sum_{N=1}^{b^j}\widetilde{\varphi}_{b,1}^{\id}\left(\frac{N}{b^j}\right)&=&b^j \left(\int_0^1 \widetilde{\varphi}_{b,1}^{\id}(x)\rd x+\frac{\overline{A}_b(j,\id)}{2}\right), \\
								\sum_{N=1}^{b^j}\widetilde{\varphi}_{b,2}^{\id}\left(\frac{N}{b^j}\right)&=&b^j\left( \int_0^1 \widetilde{\varphi}_{b,2}^{\id}(x)\rd x+\frac{\overline{A}_b(j,\id)}{2}\right),\end{eqnarray*}
						where we have
						$$ A_b(j,\id)=\begin{cases}
						                 -\frac{1}{36b^{2j}}(b^3+2b) & \mbox{if \quad} b \mbox{\quad is even}, \\
														 -\frac{1}{36b^{2j}}(b^3-b) & \mbox{if \quad} b \mbox{\quad is odd}.
						              \end{cases}$$
					and
					$$ \overline{A}_b(j,\id)=\begin{cases}
						                 -\frac{1}{36b^{2j}}(b^3-4b) & \mbox{if \quad} b \mbox{\quad is even}, \\
														 -\frac{1}{36b^{2j}}(b^3-b) & \mbox{if \quad} b \mbox{\quad is odd}.
						              \end{cases}$$
\end{lemma}

\begin{proof} We use the fact that
    \begin{equation*} \label{idphi} \varphi_{b,h}^{\id}(x)=\begin{cases}
																	(b-h)x & \mbox{if \quad} x\in \left[0,\frac{h}{b}\right], \\
																	h(1-x) & \mbox{if \quad} x\in \left[\frac{h}{b},1\right],
															\end{cases}  \end{equation*}
	and
	  \begin{equation*} \label{tauphi} \varphi_{b,h}^{\tau}(x)=\begin{cases}
																	-hx & \mbox{if \quad} x\in \left[0,\frac{b-h}{b}\right], \\
																	(b-h)x-(b-h) & \mbox{if \quad} x\in \left[\frac{b-h}{b},1\right],
															\end{cases}  \end{equation*}
															which was already mentioned in \cite{FP}.
	Let $x\in [k/b,(k+1)/b]$. Then we have $x\in[0,h/b]$ for $h\in\{k+1,\dots,b-1\}$
	and $x\in [h/b,1]$ for $h\in\{0,\dots,k\}$. We have $x\in[0,(b-h)/b]$ for
	$h\in\{0,\dots,b-k-1\}$ and $x\in[(b-h)/b,1]$ for $h\in\{b-k,\dots,b-1\}$.
	We distinguish two cases:
	\begin{enumerate}
	  \item Let $k\leq (b-1)/2$. Then we have $k\leq b-k-1$ and therefore we can write
		\begin{align*}
		     \widetilde{\varphi}_b^{\id}(x)=& \sum_{h=0}^{k}h(1-x)(-hx)+\sum_{h=k+1}^{b-k-1}(b-h)x\cdot(-hx)\\
				            &+\sum_{h=b-k}^{b-1}(b-h)x\cdot((b-h)x-(b-h)) \\
										=&\left(bk^2+bk-\frac{b^3}{6}+\frac{b}{6}\right)x^2-\left(\frac{2k^3}{3}+k^2+\frac{k}{3}\right)x=:P_k(x).
		\end{align*}
				
		\item Let $k> (b-1)/2$. Then we have $b-k-1 < k$ and therefore we can write
		\begin{align*}
		     \widetilde{\varphi}_b^{\id}(x)=& \sum_{h=0}^{b-k-1}h(1-x)(-hx)+\sum_{h=b-k}^{k}h(1-x)\cdot ((b-h)x-(b-h))\\
				            &+\sum_{h=k+1}^{b-1}(b-h)x\cdot((b-h)x-(b-h)) \\
										=&\left(bk^2-2b^2k+bk+\frac{5b^3}{6}-b^2+\frac{b}{6}\right)x^2\\
										&+\left(2b^2k-\frac{2k^3}{3}-k^2-\frac{k}{3}-b^3+b^2\right)x \\
										&+\frac{2k^3}{3}-bk^2+k^2-bk+\frac{k}{3}+\frac{b^3}{6}-\frac{b}{6}=:Q_k(x).
		\end{align*}
	\end{enumerate}
	Now we have to consider even and odd bases $b$ separately. For even $b$ we find
		\begin{align*}
		    \int_0^1  \widetilde{\varphi}_b^{\id}(x)\rd x=\sum_{k=0}^{\frac{b}{2}-1}\int_{\frac{k}{b}}^{\frac{k+1}{b}}P_k(x)\rd x+
				   \sum_{k=\frac{b}{2}}^{b-1}\int_{\frac{k}{b}}^{\frac{k+1}{b}}Q_k(x)\rd x=-\frac{1}{90b}-\frac{7b^3}{720}  
		\end{align*}
		and
		\begin{align*}
\sum_{N=1}^{b^j} \widetilde{\varphi}_b^{\id}\left(\frac{N}{b^j}\right)=&\sum_{k=0}^{\frac{b}{2}-1}\sum_{N=kb^{j-1}+1}^{(k+1)b^{j-1}}
				     P_k\left(\frac{N}{b^j}\right)+\sum_{k=\frac{b}{2}}^{b-1}\sum_{N=kb^{j-1}+1}^{(k+1)b^{j-1}}
				     Q_k\left(\frac{N}{b^j}\right)\\ 
						=&-\frac{1}{72b^j}(b^3+2b)+b^j\left(-\frac{1}{90b}-\frac{7b^3}{720}\right)\
		\end{align*}
		whereas for odd bases $b$ we compute analogously
		\begin{align*}
		    \int_0^1  \widetilde{\varphi}_b^{\id}(x)\rd x=\sum_{k=0}^{\frac{b-1}{2}}\int_{\frac{k}{b}}^{\frac{k+1}{b}}P_k(x)\rd x+
				   \sum_{k=\frac{b+1}{2}}^{b-1}\int_{\frac{k}{b}}^{\frac{k+1}{b}}Q_k(x)\rd x=\frac{7}{720b}-\frac{7b^3}{720}  
		\end{align*}
		and
		\begin{align*}\sum_{N=1}^{b^j} \widetilde{\varphi}_b^{\id}\left(\frac{N}{b^j}\right)=&\sum_{k=0}^{\frac{b-1}{2}}\sum_{N=kb^{j-1}+1}^{(k+1)b^{j-1}}
				     P_k\left(\frac{N}{b^j}\right)+\sum_{k=\frac{b+1}{2}}^{b-1}\sum_{N=kb^{j-1}+1}^{(k+1)b^{j-1}}
				     Q_k\left(\frac{N}{b^j}\right)\\ 
						=&-\frac{1}{72b^j}(b^3-b)+b^j\left(\frac{7}{720b}-\frac{7b^3}{720} \right).
		\end{align*}
	It is straightforward now to derive the claimed formula for $\sum_{N=1}^{b^j}\widetilde{\varphi}_b^{\id}\left(\frac{N}{b^j}\right)$.	Since the proofs of the other two identities may be executed analogously, we omit them at this point. 		
\end{proof}
The next lemma generalizes Lemma~\ref{id} to arbitrary permutations $\sigma\in\mathfrak{S}_b$. The main idea of the proof is to reduce the case of general permutations $\sigma$ to the case where $\sigma=\id$. The latter case has been analyzed in the previous lemma already. We advise the reader to consult also the proof of \cite[Lemma 4]{FPPS09}, since we follow closely the ideas there.

\begin{lemma} \label{generalsigma} Let $\sigma\in\mathfrak{S}_b$. For all $j\in\{1,\dots,n\}$ we have
   \begin{eqnarray*}\sum_{N=1}^{b^j}\widetilde{\varphi}_b^{\sigma}\left(\frac{N}{b^j}\right)&=&b^j \left(\int_0^1 \widetilde{\varphi}_b^{\sigma}(x)\rd x+\frac{A_b(j,\id)}{2}\right), \\
	              \sum_{N=1}^{b^j}\widetilde{\varphi}_{b,1}^{\sigma}\left(\frac{N}{b^j}\right)&=&b^j\left( \int_0^1 \widetilde{\varphi}_{b,1}^{\sigma}(x)\rd x+\frac{\overline{A}_b(j,\id)}{2}\right), \\
								\sum_{N=1}^{b^j}\widetilde{\varphi}_{b,2}^{\sigma}\left(\frac{N}{b^j}\right)&=&b^j \left(\int_0^1 \widetilde{\varphi}_{b,2}^{\sigma}(x)\rd x+\frac{\overline{A}_b(j,\id)}{2}\right),\end{eqnarray*}
						where $A_b(j,\id)$ and $\overline{A}_b(j,\id)$ are as defined in Lemma~\ref{id}.
\end{lemma}

\begin{proof} It was shown in the proof of \cite[Lemma 4]{FPPS09}, by applying Simpson's quadrature rule, that
    $$ \sum_{N=1}^{b^j}\varphi_b^{\sigma,(2)}\left(\frac{N}{b^j}\right)=b^j\left(\int_0^1 \varphi_b^{\sigma,(2)}(x)\rd x+\frac{A_b(j,\sigma)}{2}\right) $$
    with
		   $$ A_b(j,\sigma)=\frac{1}{6b^{2j}}\sum_{k=1}^{b}\left(\big(\varphi_b^{\sigma,(2)}\big)'\left(\frac{k}{b}-0\right)-\big(\varphi_b^{\sigma,(2)}\big)'\left(\frac{k}{b}+0\right)\right), $$
		where here and later on by $f'(x-0)$ we mean the left-derivative and by $f'(x+0)$ the right-derivative of the function
$f$ at $x$ .
    The only properties of $\varphi_b^{\sigma,(2)}$ the authors needed to show this identity are the fact that $\varphi_b^{\sigma,(2)}$ is quadratic on intervals $[k/b,(k+1)/b]$ as well as the $1$-periodicity of this function. Since $\widetilde{\varphi}_b^{\sigma}$ has these two properties as well, an analogous relation is also true for $\widetilde{\varphi}_b^{\sigma}$. 
    Now we need the definition $\widetilde{\varphi}_b^{\sigma}=\sum_{h=0}^{b-1}\varphi_{b,h}^{\sigma}\varphi_{b,h}^{\overline{\sigma}}$ to deduce
    \begin{align*}
       &\big(\widetilde{\varphi}_b^{\sigma}\big)'\left(\frac{k}{b}-0\right)-\big(\widetilde{\varphi}_b^{\sigma}\big)'\left(\frac{k}{b}+0\right) \\  =&\sum_{h=0}^{b-1}\left\{\varphi_{b,h}^{\sigma}\left(\frac{k}{b}\right)\big(\varphi_{b,h}^{\overline{\sigma}}\big)'\left(\frac{k}{b}-0\right)+\varphi_{b,h}^{\overline{\sigma}}\left(\frac{k}{b}\right)\big(\varphi_{b,h}^{\sigma}\big)'\left(\frac{k}{b}-0\right) \right. \\ &\hspace{1cm}- \left.  \varphi_{b,h}^{\sigma}\left(\frac{k}{b}\right)\big(\varphi_{b,h}^{\overline{\sigma}}\big)'\left(\frac{k}{b}+0\right)-\varphi_{b,h}^{\overline{\sigma}}\left(\frac{k}{b}\right)\big(\varphi_{b,h}^{\sigma}\big)'\left(\frac{k}{b}+0\right)
       \right\} \\ =&\sum_{h=0}^{b-1}\varphi_{b,h}^{\sigma}\left(\frac{k}{b}\right)\left(\big(\varphi_{b,h}^{\overline{\sigma}}\big)'\left(\frac{k-1}{b}+0\right)-\big(\varphi_{b,h}^{\overline{\sigma}}\big)'\left(\frac{k}{b}+0\right)\right) \\
&+\sum_{h=0}^{b-1}\varphi_{b,h}^{\overline{\sigma}}\left(\frac{k}{b}\right)\left(\big(\varphi_{b,h}^{\sigma}\big)'\left(\frac{k-1}{b}+0\right)-\big(\varphi_{b,h}^{\sigma}\big)'\left(\frac{k}{b}+0\right)\right)=S_1+S_2.
    \end{align*} 
We define $f_{h,k}:=\big(\varphi_{b,h}^{\sigma}\big)'\left(\frac{k}{b}+0\right)$ and $\overline{f}_{h,k}:=\big(\varphi_{b,h}^{\overline{\sigma}}\big)'\left(\frac{k}{b}+0\right)$.
From the linearity of $\varphi_{b,h}^{\sigma}$ and $\varphi_{b,h}^{\overline{\sigma}}$ on $[k/b,(k+1)/b]$ we have
$\varphi_{b,h}^{\sigma}(k/b)=\int_0^{k/b}\big(\varphi_{b,h}^{\sigma}\big)'(x)\rd x=\frac{1}{b}\sum_{l=0}^{k-1}f_{h,l}$ and
also $\varphi_{b,h}^{\overline{\sigma}}(k/b)=\frac{1}{b}\sum_{l=0}^{k-1}\overline{f}_{h,l}$. For $k=b$, this yields $\sum_{l=0}^{b-1}f_{h,l}=\sum_{l=0}^{b-1}\overline{f}_{h,l}=0$. Hence for every $h\in\{0,\dots,b-1\}$ we have
$$ \sum_{k=1}^b S_1=\frac{1}{b}\sum_{l=0}^{b-1}f_{h,l}\sum_{k=l+1}^{b}(\overline{f}_{h,k-1}-\overline{f}_{h,k})=\frac{1}{b}\sum_{l=0}^{b-1}f_{h,l}\overline{f}_{h,l} $$
and analogously
$$ \sum_{k=1}^b S_2=\frac{1}{b}\sum_{l=0}^{b-1}\overline{f}_{h,l}\sum_{k=l+1}^{b}(f_{h,k-1}-f_{h,k})=\frac{1}{b}\sum_{l=0}^{b-1}\overline{f}_{h,l}f_{h,l}=\sum_{k=1}^b S_1. $$
Finally we conclude
\begin{align*}
     A_b(j,\sigma)=&\frac{1}{3b^{2j}}\sum_{h=0}^{b-1}\sum_{l=0}^{b-1}f_{h,l}\overline{f}_{h,l}             =\frac{1}{3b^{2j}}\sum_{h=0}^{b-1}\sum_{l=0}^{b-1}\left(\big(\varphi_{b,h}^{\sigma}\big)'\left(\frac{l}{b}+0\right)\big(\varphi_{b,h}^{\overline{\sigma}}\big)'\left(\frac{l}{b}+0\right)\right) \\ &=\frac{1}{3b^{2j}}\sum_{h=0}^{b-1}\sum_{l=0}^{b-1}\left(\big(\varphi_{b,h}^{\id}\big)'\left(\frac{\sigma(l)}{b}+0\right)\big(\varphi_{b,h}^{\tau}\big)'\left(\frac{\sigma(l)}{b}+0\right)\right) \\
&=\frac{1}{3b^{2j}}\sum_{h=0}^{b-1}\sum_{l=0}^{b-1}\left(\big(\varphi_{b,h}^{\id}\big)'\left(\frac{l}{b}+0\right)\big(\varphi_{b,h}^{\tau}\big)'\left(\frac{l}{b}+0\right)\right)=A_b(j,\id),   
\end{align*}
where we used the relations
\begin{equation*} \label{vertid} \big(\varphi_{b,h}^{\sigma}\big)'\left(\frac{l}{b}+0\right)=\big(\varphi_{b,h}^{\id}\big)'\left(\frac{\sigma(l)}{b}+0\right)  \end{equation*}
and
\begin{equation*} \label{verttau} \big(\varphi_{b,h}^{\overline{\sigma}}\big)'\left(\frac{l}{b}+0\right)=\big(\varphi_{b,h}^{\tau}\big)'\left(\frac{\sigma(l)}{b}+0\right).  \end{equation*}
They follow both directly from the definition of $\varphi_{b,h}^{\sigma}$. The first relation has also been used in \cite{FP,FPPS09}. The proof of the first claim of this lemma is complete. Since the other two identities may be proven completely analogously, we omit an explicit proof.
\end{proof}

Now we are ready to show the main lemma of this paper. We will combine Lemmas~\ref{ugly},~\ref{id} and~\ref{generalsigma} to obtain this result. 

\begin{lemma} \label{main} Let $\sigma\in\mathfrak{S}_b$. Then we have for even bases $b$
\begin{align*}
  \frac{2}{b^{2n}}&\sum_{j=1}^{n}\sum_{\lambda,N=1}^{b^n}\varphi_{b,\varepsilon_j(\lambda,N,\Sigma)}^{\sigma}\left(\frac{N}{b^j}\right)\varphi_{b,\varepsilon_j(\lambda,N,\Sigma^{\ast})}^{\overline{\sigma}}\left(\frac{N}{b^j}\right) \\
	 =& n\left(\widetilde{\Phi}_b^{\sigma}+\frac12 \widetilde{\Phi}_{b,1}^{\sigma}+\frac12 \widetilde{\Phi}_{b,2}^{\sigma} \right)
	+\left(\widetilde{\Phi}_b^{\sigma}-\frac12\widetilde{\Phi}_{b,1}^{\sigma}-\frac12\widetilde{\Phi}_{b,2}^{\sigma}\right)-\frac{1}{36}-\frac{1}{18b^{2n}}.
\end{align*}
and for odd bases $b$
\begin{align*} \frac{2}{b^{2n}}&\sum_{j=1}^{n}\sum_{\lambda,N=1}^{b^n}\varphi_{b,\varepsilon_j(\lambda,N,\Sigma)}^{\sigma}\left(\frac{N}{b^j}\right)\varphi_{b,\varepsilon_j(\lambda,N,\Sigma^{\ast})}^{\overline{\sigma}}\left(\frac{N}{b^j}\right) \\
	 =& n\left(\widetilde{\Phi}_b^{\sigma}+\frac12 \widetilde{\Phi}_{b,1}^{\sigma}+\frac12 \widetilde{\Phi}_{b,2}^{\sigma} \right)+\left(-\frac{1}{36}+\frac{b^2}{b^2-1}\left(\widetilde{\Phi}_b^{\sigma}-\frac12\widetilde{\Phi}_{b,1}^{\sigma}-\frac12 \widetilde{\Phi}_{b,2}^{\sigma}\right)\right)\left(1-\frac{1}{b^{2n}}\right).
\end{align*}
\end{lemma}

\begin{proof} At first we remark that for $N=N_{n-1}b^{n-1}+\dots +N_0$ the number $\nu_j(N,\Sigma)$ depends only on the digits 
$N_j,\dots,N_{n-1}$, which follows directly from its definition in Lemma~\ref{allgemein}. On the other hand, the values
of $\widetilde{\varphi}_b^{\sigma}\left(\frac{N}{b^j}\right)$, $\widetilde{\varphi}_{b,1}^{\sigma}\left(\frac{N}{b^j}\right)$ and $\widetilde{\varphi}_{b,2}^{\sigma}\left(\frac{N}{b^j}\right)$ depend only on the digits $N_0,\dots,N_{j-1}$. This can be seen from the 1-periodicity of these functions, since
\begin{align*} \widetilde{\varphi}_b^{\sigma}\left(\frac{N}{b^j}\right)=&\widetilde{\varphi}_b^{\sigma}\left(\left\{\frac{N}{b^j}\right\}\right)=\widetilde{\varphi}_b^{\sigma}\left(\left\{N_{n-1}b^{n-j-1}+\dots+N_j+N_{j-1}b^{-1}+\dots+N_0b^{-j}\right\}\right)\\ =& \widetilde{\varphi}_b^{\sigma}\left(N_{j-1}b^{-1}+\dots+N_0b^{-j}\right)
\end{align*}
and analogously for $\widetilde{\varphi}_{b,1}^{\sigma}$ and $\widetilde{\varphi}_{b,2}^{\sigma}$.
We set $f_b(j):=\left\lfloor(b^{n-j}-1)/2\right\rfloor$. Lemma~\ref{ugly} leads to
\begin{align*}
   \sum_{j=1}^{n}&\sum_{\lambda,N=1}^{b^n}\varphi_{b,\varepsilon_j(\lambda,N,\Sigma)}^{\sigma}\left(\frac{N}{b^j}\right)\varphi_{b,\varepsilon_j(\lambda,N,\Sigma^{\ast})}^{\overline{\sigma}}\left(\frac{N}{b^j}\right) \\
	=&\sum_{j=1}^{n-1} \left\{ \sum_{\ell=0}^{f_b(j)}\sum_{\substack{N_j,\dots,N_{n-1}=0 \\ \nu_j(N,\Sigma)=\ell}}^{b-1}
	     \sum_{N_0,\dots,N_{j-1}=0}^{b-1}\Bigg(b^{n-1}\widetilde{\varphi}_b^{\sigma}\left(\frac{N}{b^j}\right) \right. \\ &\hspace{5cm}+ \left.b^{j-1}(b^{n-j}-1-2\ell)\left(\widetilde{\varphi}_{b,1}^{\sigma}\left(\frac{N}{b^j}\right)-\widetilde{\varphi}_b^{\sigma}\left(\frac{N}{b^j}\right)\right)\Bigg) \right. \\ &+ \left. \sum_{\ell=f_b(j)+1}^{b^{n-j}-1}\sum_{\substack{N_j,\dots,N_{n-1}=0 \\ \nu_j(N,\Sigma)=\ell}}^{b-1}
	     \sum_{N_0,\dots,N_{j-1}=0}^{b-1}\Bigg(b^{n-1}\widetilde{\varphi}_b^{\sigma}\left(\frac{N}{b^j}\right)\right. \\ &\hspace{5cm}+ \left.b^{j-1}(2\ell+1-b^{n-j})\left(\widetilde{\varphi}_{b,2}^{\sigma}\left(\frac{N}{b^j}\right)-\widetilde{\varphi}_b^{\sigma}\left(\frac{N}{b^j}\right)\right)\Bigg) \right\} \\
		&+ \sum_{N=1}^{b^n}b^{n-1}\widetilde{\varphi}_b^{\sigma}\left(\frac{N}{b^n}\right) \\
	=&\sum_{j=1}^{n-1} \left\{ \sum_{\ell=0}^{f_b(j)}
	     \sum_{N=1}^{b^j}\Bigg(b^{n-1}\widetilde{\varphi}_b^{\sigma}\left(\frac{N}{b^j}\right)+b^{j-1}(b^{n-j}-1-2\ell)\left(\widetilde{\varphi}_{b,1}^{\sigma}\left(\frac{N}{b^j}\right)-\widetilde{\varphi}_b^{\sigma}\left(\frac{N}{b^j}\right)\right)\Bigg) \right. \\ &+ \left. \sum_{\ell=f_b(j)+1}^{b^{n-j}-1}
	     \sum_{N=1}^{b^j}\Bigg(b^{n-1}\widetilde{\varphi}_b^{\sigma}\left(\frac{N}{b^j}\right)+b^{j-1}(2\ell+1-b^{n-j})\left(\widetilde{\varphi}_{b,2}^{\sigma}\left(\frac{N}{b^j}\right)-\widetilde{\varphi}_b^{\sigma}\left(\frac{N}{b^j}\right)\right)\Bigg) \right\} \\
		&+ \sum_{N=1}^{b^n}b^{n-1}\widetilde{\varphi}_b^{\sigma}\left(\frac{N}{b^n}\right) 
\end{align*}
At this point we need to treat the cases of even and odd bases $b$ separately. Let us first consider even bases.
Then we have $f_b(j)=b^{n-j}/2-1$. With Lemma~\ref{generalsigma} we get
\begin{align*}
   \sum_{j=1}^{n}&\sum_{\lambda,N=1}^{b^n}\varphi_{b,\varepsilon_j(\lambda,N,\Sigma)}^{\sigma}\left(\frac{N}{b^j}\right)\varphi_{b,\varepsilon_j(\lambda,N,\Sigma^{\ast})}^{\overline{\sigma}}\left(\frac{N}{b^j}\right) \\
	=&\sum_{j=1}^{n-1} \left\{ b^{2n-j-1} b^j\left(b\widetilde{\Phi}_b^{\sigma}+\frac{A_b(j,\id)}{2}\right) \right. \\
	  &+ \left. b^{j-1}\sum_{\ell=0}^{b^{n-j}/2-1}
	    (b^{n-j}-1-2\ell)b^j\left(b\widetilde{\Phi}_{b,1}^{\sigma}+\frac{\overline{A}_b(j,\id)}{2}-b\widetilde{\Phi}_{b}^{\sigma}-\frac{A_b(j,\id)}{2}\right) \right. \\ &+  \left. b^{j-1}\sum_{\ell=b^{n-j}/2}^{b^{n-j}-1}
	    (2\ell+1-b^{n-j})b^j\left(b\widetilde{\Phi}_{b,2}^{\sigma}+\frac{\overline{A}_b(j,\id)}{2}-b\widetilde{\Phi}_{b}^{\sigma}-\frac{A_b(j,\id)}{2}\right) \right\} \\
		&+ \sum_{N=1}^{b^n}b^{n-1}\widetilde{\varphi}_b^{\sigma}\left(\frac{N}{b^n}\right) \\
	=&\sum_{j=1}^{n-1} \left\{ b^{2n-1}\left(b\widetilde{\Phi}_b^{\sigma}-\frac{1}{72b^{2j}}(b^3+2b)\right) + \frac14 b^{2n-1}\left(b\widetilde{\Phi}_{b,1}^{\sigma}-b\widetilde{\Phi}_{b}^{\sigma}+\frac{1}{12b^{2j-1}}\right) \right. \\ &+  \left. \frac14 b^{2n-1}\left(b\widetilde{\Phi}_{b,2}^{\sigma}-b\widetilde{\Phi}_{b}^{\sigma}+\frac{1}{12b^{2j-1}}\right) \right\} + b^{2n-1}\left(b\widetilde{\Phi}_b^{\sigma}-\frac{1}{72b^{2n}}(b^3+2b)\right). \\
\end{align*}
Now a straightforward calculation yields the claimed result for even bases $b$. For odd bases $b$ we have
$f_b(j)=(b^{n-j}-1)/2$ and hence we obtain similarly as above
\begin{align*}
   \sum_{j=1}^{n}&\sum_{\lambda,N=1}^{b^n}\varphi_{b,\varepsilon_j(\lambda,N,\Sigma)}^{\sigma}\left(\frac{N}{b^j}\right)\varphi_{b,\varepsilon_j(\lambda,N,\Sigma^{\ast})}^{\overline{\sigma}}\left(\frac{N}{b^j}\right) \\
	=&\sum_{j=1}^{n-1} \left\{ b^{2n-j-1} b^j\left(b\widetilde{\Phi}_b^{\sigma}+\frac{A_b(j,\id)}{2}\right) + b^{j-1}\sum_{\ell=0}^{(b^{n-j}-1)/2}
	    (b^{n-j}-1-2\ell)b^j\left(b\widetilde{\Phi}_{b,1}^{\sigma}-b\widetilde{\Phi}_{b}^{\sigma})\right) \right. \\ &+  \left. b^{j-1}\sum_{\ell=(b^{n-j}+1)/2}^{b^{n-j}-1}
	    (2\ell+1-b^{n-j})b^j\left(b\widetilde{\Phi}_{b,2}^{\sigma}-b\widetilde{\Phi}_{b}^{\sigma}\right) \right\} 
		+ \sum_{N=1}^{b^n}b^{n-1}\widetilde{\varphi}_b^{\sigma}\left(\frac{N}{b^n}\right) \\
	=&\sum_{j=1}^{n-1} \left\{ b^{2n-1}\left(b\widetilde{\Phi}_b^{\sigma}-\frac{1}{72b^{2j}}(b^3-b)\right) + \frac{1}{4b}(b^{2n}-b^{2j})\left(b\widetilde{\Phi}_{b,1}^{\sigma}-b\widetilde{\Phi}_{b}^{\sigma}\right) \right. \\ &+  \left. \frac{1}{4b}(b^{2n}-b^{2j})\left(b\widetilde{\Phi}_{b,2}^{\sigma}-b\widetilde{\Phi}_{b}^{\sigma}\right) \right\} + b^{2n-1}\left(b\widetilde{\Phi}_b^{\sigma}-\frac{1}{72b^{2n}}(b^3-b)\right). 
\end{align*}
The rest of the proof is again a matter of elementary calculations.
\end{proof}

\begin{remark} \label{simpleform} \rm Tedious computations, similar to those we needed to prove Lemma~\ref{altern}, yield for $\sigma\in\mathcal{A}_b(\tau)$ the relation
  $$ \widetilde{\Phi}_b^{\sigma}-\frac12\widetilde{\Phi}_{b,1}^{\sigma}-\frac12 \widetilde{\Phi}_{b,2}^{\sigma}
       = \begin{cases}
                -\frac{1}{24} & \mbox{if  } b \mbox{  is even  }, \\
                -\frac{1}{24}\frac{b^2-1}{b^2} & \mbox{if  } b \mbox{  is odd  }.
         \end{cases} $$
     In this case, Lemma~\ref{main} can be displayed in a much simplier form, namely
     \begin{align*} \frac{2}{b^{2n}}&\sum_{j=1}^{n}\sum_{\lambda,N=1}^{b^n}\varphi_{b,\varepsilon_j(\lambda,N,\Sigma)}^{\sigma}\left(\frac{N}{b^j}\right)\varphi_{b,\varepsilon_j(\lambda,N,\Sigma^{\ast})}^{\overline{\sigma}}\left(\frac{N}{b^j}\right) \\
     &=n\left(\widetilde{\Phi}_b^{\sigma}+\frac12 \widetilde{\Phi}_{b,1}^{\sigma}+\frac12 \widetilde{\Phi}_{b,2}^{\sigma} \right)
         -\frac{5}{72}+\frac{1-9\cdot (-1)^b}{144b^{2n}}. \end{align*}
\end{remark}

\section{Numerical results} \label{numerical}

We avoid all the proofs in this section, since they contain elementary, but very lengthy and technical calculations. \\
The constant $c_b^{\sigma}$ which appears in Theorem~\ref{theo} is rather hard to compute. We therefore present an alternative formula in the subsequent lemma.

\begin{lemma} \label{altern}
   Let $n\in\NN$, $\sigma\in\mathcal{A}_b(\tau)$ and $\Sigma\in\{\sigma,\osigma\}^n$. Then we have
	 $$ \lim_{n\to\infty}\frac{2b^n\,L_2(\cR_{b,n}^{\Sigma,\sym})}{\sqrt{\log{(2b^n)}}}=\sqrt{\frac{c_b^{\sigma}}{\log{b}}}, $$
	where
	\begin{align*}
	   c_b^{\sigma}=&2\Phi_b^{\sigma,(2)}+\widetilde{\Phi}_b^{\sigma}+\frac12 \widetilde{\Phi}_{b,1}^{\sigma}+\frac12 \widetilde{\Phi}_{b,2}^{\sigma} \\ =&\frac{16-12b-111b^2+228b^3-112b^4}{72b^2}-\frac{1-(-1)^b}{16b^3} \\
		       &+\frac{4}{b^3}\sum_{k_1,k_2=0}^{b-1}\max(\sigma(k_1),\sigma(k_2))\left(\frac{b}{2}\big(\max(k_1,k_2)+\max(k_1+k_2,b-1)\big)-k_1^2-k_1\right). 
	\end{align*}
\end{lemma}

From this result we can deduce the following rule, which states that the constant $c_b^{\sigma}$ is invariant with respect to switching two complementary elements in the permutation $\sigma$.

\begin{corollary} \label{vertauschungsregel} Let $\sigma\in\mathcal{A}_b(\tau)$ and $d\in\{0,\dots,b-1\}$. Then we define the permutation $\widehat{\sigma}\in\mathcal{A}_b(\tau)$ in the following way: For $k\in\{0,\dots,b-1\}\setminus\{d,b-1-d\}$ we set $\widehat{\sigma}(k)=\sigma(k)$ and additionally we set $\widehat{\sigma}(d)=\sigma(b-1-d)$ and $\widehat{\sigma}(b-1-d)=\sigma(d)$. Then we have $c_b^{\sigma}=c_b^{\widehat{\sigma}}$.
\end{corollary}

Now we would like to find for each base $b$ the permutation $\sigma_{b}^{\min}\in\mathcal{A}_b(\tau)$ for which the constant $c_b^{\sigma}$ becomes minimal. We therefore employ computer search algorithms. Corollary~\ref{vertauschungsregel} allows us to reduce the number of permutations we have to check significantly. We do not have to check every single permutation that is contained in $\mathcal{A}_b(\tau)$, but only those which are elements of the subset
 $$ \mathcal{B}_b(\tau):=\left\{\sigma\in\mathcal{A}_b(\tau): \sigma(k)\in \big\{0,1,\dots, \big\lfloor \frac{b-1}{2}\big\rfloor\big\} \text{ for all } k\in\big\{0,1,\dots, \big\lfloor \frac{b-1}{2}\big\rfloor\big\}\right\}. $$
That means we have to check $\lfloor b/2 \rfloor !$ permutations instead of $2^{\lfloor b/2 \rfloor}\lfloor b/2 \rfloor !$ to find the minimum value for $c_b^{\sigma}$. Our numerical investigations show that there are often several permutations $\sigma\in\mathcal{B}_b(\tau)$ where the minimum value for $c_b^{\sigma}$ is attained. Table~\ref{table1} lists for each base $b\in\{2,\dots,27\}$ one permutation $\sigma\in\mathcal{B}_b(\tau)$ where $c_b^{\sigma}$ is minimal and the number $g_b$ of permutations in $\mathcal{B}_b(\tau)$ which give the minimal value for $c_b^{\sigma}$. Then there are $2^{\lfloor b/2 \rfloor}g_b$ permutations in $\mathcal{A}_b(\tau)$ which yield the lowest constant in each base. Of course, we also present the corresponding values for $c_b^{\sigma}$ and $\sqrt{c_b^{\sigma}/\log{b}}$. Since the permutations in $\mathcal{B}_b(\tau)$ are completely determined by the permutation of the digits $0,1,\dots,\lfloor (b-1)/2 \rfloor$, we only give these partial permutations in Table~\ref{table1}. We use the usual cycle notation. For instance, the permutation $(0,1,2)\in\mathcal{B}_7(\tau)$ on the set $\{0,1,\dots,6\}$ is given by $\sigma(0)=1$, $\sigma(1)=2$ and $\sigma(2)=0$. The values of $\sigma(3)$, $\sigma(4)$, $\sigma(5)$ and $\sigma(6)$ can then be obtained through the relation $\sigma(6-k)=6-\sigma(k)$ for $k=0,1,2,3$. 

\begin{table}[ht]
\centering
\begin{tabular}[h]{ |c| c| c |c| c| } \hline 
  $\mathbf{b}$ & $\mathbf{\sigma_b^{\min}}$ & $\mathbf{g_b}$ & $\mathbf{c_b^{\sigma}}$ & $\mathbf{\sqrt{\frac{c_b^{\sigma}}{\log{b}}}}$  \\ \hline
  2 & $\id$ & 1 & 1/24 & 0.245178 \\ 
  3 & $\id$ & 1 & 5/81 &   0.237039 \\ 
	4 & $\id$ & 2 & 1/12 &   0.245178 \\ 
	5 & $(0,1)$ & 1 & 29/375 &  0.219202 \\
	6 & $(0,1)$ & 4 & 67/648 &  0.240220 \\ 
	7 & $(0,1,2)$ & 2 & 2/21 &   0.221229 \\ 
	8 & $(0,2,3,1)$ & 2 & 3/32 &   0.212330 \\ 
	9 & $(0,1,3)$ & 4 & 26/243 &   0.220671 \\ 
	10 & $(0,3,4,1)$ & 2 & 111/1000 &   0.219560 \\ 
	11 & $(0,2)(1,4)$ & 1 & 415/3993 &   0.208189 \\ 
	12 & $(0,3)(2,5)$ & 2 & 35/324 &   0.208500 \\ 
	13 & $(0,2)(1,5)(3,4)$ & 1 & 55/507 &  0.205654 \\ 
	14 & $(0,2)(1,5)(4,6)$ & 2 & 983/8232 & 0.212715 \\ 
	15 & $(0,4)(2,6)$ & 3 & 236/2025 &   0.207450 \\ 
	16 & $(0,5,4)(2,3,7)$ & 4 & 23/192 &   0.207859 \\ 
	17 & $(0,3,5,6,4,2)(1,7)$ & 2 & 584/4913 & 0.204829 \\ 
	18 & $(0,5,8,3)(1,2,7,6)$ & 2 & 241/1944 &  0.207101 \\
	19 & $(0,5)(2,8)(4,6,7)$ & 2 & 827/6859 &  0.202358 \\ 
	20 & $(0,2,4)(1,8)(3,6)(5,7,9)$ & 8 & 193/1500 &   0.207243 \\ 
	21 & $(0,6)(2,9)(5,8)$ & 1 & 491/3969 &  0.201576 \\ 
	22 & $(0,4,2,1,9,8,5,6,10,3,7)$ & 8 & 4219/31944 &   0.206708 \\ 
	23 & $(0,6)(2,10)(4,8)(7,9)$ & 1 & 4586/36501 &   0.200175 \\ 
	24 & $(0,7,11,3,5,8,1,2,10,9,6,4)$ & 16 & 343/2592 &   0.204055 \\ 
	25 & $(0,4,6,8,10,7)(1,9,5,3,11,2)$ & 8 & 1234/9375 &   0.202218 \\ 
	26 & $(0,7,12,5)(1,2,11,10)(3,4,9,8)$ & 2 & 2236/17576 &   0.198792 \\ 
	27 & $(0,3,1,10,6,8,11,9,4,12,2,7)$ & 14 & 289/2187 &   0.200235 \\ \hline
\end{tabular}
\caption{Numerical results for the full search in $\mathcal{B}_b(\tau)$}
\label{table1}
\end{table}

Finally, we would like to explain how the algorithm we used to create the results in Table~\ref{table1} works. 
At first we define several global variables. 

\begin{algorithmic} \tt
\State double $\mathtt{b}$
\State double $\mathtt{min}$
\State $\mathtt{byte}[\,\,]$ $\mathtt{sigma}$
\State List<$\mathtt{byte}[\,\,]$> $\mathtt{list}$
\end{algorithmic}

In the function $\textsc{Minimum}$, the variable $b$ gets the value of the base the user enters. For the variable $\min$, we 
choose the largest possible integer value. The variable $m$ gives the length of the permutation we
would like to create. Then we create a new array $\mathrm{sigma}$ of length $n$ and set $\mathrm{sigma}[i]=i$
for all $i\in \{0,1,\dots,m-1\}$. We also generate a list, in which the permutations with minimal constant $c_b^{\sigma}$ are stored. Finally, we call the function $\textsc{Perm}$ with the value $\lfloor b/2 \rfloor-1$. 

\begin{algorithmic} \tt
\Function{Minimum()}{} 
 \State $\mathtt{b \gets base}$
 \State $\mathtt{min\gets 2^{31}-1}$
 \State int $\mathtt{m\gets \lfloor b/2 \rfloor} $ 
       \If {$\mathtt{Mod[b,2]==1}$}
          \State $\mathtt{m\gets m+1}$
        \EndIf
 \State $\mathtt{sigma \gets \text{ new } \mathtt{byte}[m]}$
 \State $\mathtt{list \gets \text{ new } \mathtt{ArrayList<>(\,\,)}}$
 \For{int $\mathtt{i:=0 \to m-1; \,\, i{+}{+}}$}
    \State $\mathtt{sigma[i]=i}$ 
 \EndFor
 \State $\mathtt{\Call{Perm}{\lfloor b/2 \rfloor-1}}$
\EndFunction
\end{algorithmic}

The function $\textsc{Perm}$ calculates all valid permutations and inserts
them into the formula for $c_b^{\sigma}$ given in Lemma~\ref{altern}. It uses
the function $\textsc{Sum}$, which is explained below. Whenever
a new minimal value for $c_b^{\sigma}$ is reached, this value is stored in a list, while the
list of the previous minimums is emptied (list.clear($\mathrm{sigma}$)). In case that the new calculated value
matches the current minimal value, it is added to the list of minimums (list.add($\mathrm{sigma}$)).
The permutations are created by calling the function $\textsc{Perm}$ recursively. The function $\textsc{Swap}(i,m)$
exchanges $\mathrm{sigma}[i]$ and $\mathrm{sigma}[m]$ in the array $\mathrm{sigma}$.

\begin{algorithmic}
\Function{Perm}{$\mathtt{m}$} 
   \If {$\mathtt{m==0}$}
    \State double $\mathtt{result \gets \frac{16-12b-111b^2+228b^3-112b^4}{72b^2}-\frac{1-(-1)^b}{16b^3}+\frac{4}{b^3}\Call{Sum()}{}}$
    \If {$\mathtt{result<\mathtt{min}}$}
        \State $\mathtt{min \gets result}$
        \State list.clear()
    \EndIf
    \If {$\mathtt{result==\mathtt{min}}$}
        \State list.add($\mathtt{sigma}$)
    \EndIf
\Else
  \State $\mathtt{\Call{Perm}{m-1}}$
  \For{int $\mathtt{i:=0 \to m-1; \,\, i{+}{+}}$}  
    \State $\mathtt{\Call{Swap}{i,m}}$
    \State $\mathtt{\Call{Perm}{m-1}}$
    \State $\mathtt{\Call{Swap}{i,m}}$
  \EndFor  
\EndIf
\EndFunction
\end{algorithmic}

\begin{algorithmic}

\Function{Swap}{$i,m$} 
   \State int $\mathtt{temp}\gets \mathtt{sigma[i]}$
   \State $\mathtt{sigma[i]}\gets \mathtt{sigma[m]}$
   \State $\mathtt{sigma[m]}\gets \mathtt{temp}$
\EndFunction
\end{algorithmic}

The sum in the formula for $c_b^{\sigma}$ given in Lemma~\ref{altern} is computed with two \textbf{for}-loops. 
We distinguish four cases to calculate $\max(\sigma(k_1),\sigma(k_2))$ effectively.

\begin{algorithmic}
\Function{Sum()}{} 
 \State double $\mathtt{sum \gets 0}$
 \For{int $\mathtt{k_1:=0 \to b-1; \,\, i{+}{+}}$}   \For{int $\mathtt{k_2:=0 \to b-1; \,\, i{+}{+}}$}  
    \State double $\mathtt{max \gets 0}$
    \If {$\mathtt{k_1\geq m}$ and $\mathtt{k_2\geq m}$}
       \State $\mathtt{max \gets \max(b-1-\mathtt{sigma}[b-1-k_1],b-1-\mathtt{sigma}[b-1-k_2])}$
      \Else
          \If {$\mathtt{k_1\geq m}$ and $\mathtt{k_2< m}$}
            \State $\mathtt{max \gets b-1-\mathtt{sigma}[b-1-k_1]}$
               \Else
                  \If {$\mathtt{k_1< m}$ and $\mathtt{k_2\geq m}$}
                     \State $\mathtt{max \gets b-1-\mathtt{sigma}[b-1-k_2]}$
                        \Else
                           \State $\mathtt{max \gets \max(\mathtt{sigma}[k_1],\mathtt{sigma}[k_2])}$
                  \EndIf
          \EndIf
    \EndIf
    \State $\mathtt{sum \gets \mathtt{sum}+\big(\frac{b}{2}(\max(k_1,k_2)+\max(k_1+k_2,b-1))-k_1^2-k_1\big)\cdot \mathtt{max}}$
  \EndFor \EndFor
  \State \Return $\mathtt{sum}$
\EndFunction

\end{algorithmic}

\section{Conclusions} \label{conc}

We should compare our numerical results to those in Section 5 of \cite{FPPS09}. There the authors searched for the best permutations $\sigma\in\mathcal{A}_b(\tau)$ to obtain a minimal $L_2$ discrepancy of the digit scrambled Hammersley point sets $\cR_{b,n}^{\Sigma}$, where $\Sigma\in\{\sigma,\osigma\}^n$.
The authors obtained the lowest $L_2$ discrepancy overall in base $22$; the corresponding leading constant is $0.179069...$ This number has 
already been mentioned in the introduction. We obtain the lowest leading constant for $L_2(\cR_{b,n}^{\Sigma,\sym})$ in base $26$, namely $0.198792...$. In general, the minimal constants of the symmetrized Hammersley point sets considered in this work are slightly higher
than the minimal constants of the digit scrambled Hammersley point sets in every base, at least up to base $23$ (Table 1 in \cite{FPPS09} ends after this base). The advantage of the symmetrized point sets is the fact that we do not have to care about the arrangement of $\sigma$ and $\osigma$ in $\Sigma$ (see Remark~\ref{arrang}). As for $L_2(\cR_{b,n}^{\Sigma})$, this is the case if and only if $\Phi_b^{\sigma}=0$ (see \cite[Table 2]{FPPS09}). Additionally, Example~\ref{example} indicates that the values of $c_b^{\sigma}$ for "good" and "bad" permutations do not spread so much for the symmetrized point sets as it is the case for
the digit scrambled point sets. Indeed, our numerical results suggest that for even bases the highest value for $c_b^{\sigma}$ of all $\sigma\in\mathcal{B}_b(\tau)$ is always attained only for the identity and the permutation which is determined through the relations $\sigma(k)=b/2-1-k$ for all $k\in\{0,\dots,b/2-1\}$ and that for odd bases $\max_{\sigma \in \mathcal{B}_b(\tau)}c_b^{\sigma}$ is attained if and only if $\sigma=\id$. It remains an unresolved question if there exists a positive absolute constant $C\in\RR$ such that
$ \min_{\sigma \in \mathcal{B}_b(\tau)}c_b^{\sigma} \leq C $
for all bases $b\geq 2$.

\bigskip

\noindent {\bf Acknowledgments.} The first author is supported by the Austrian Science Fund (FWF): Project F5509-N26, which is a part of the Special Research Program "Quasi-Monte Carlo Methods: Theory and Applications".

\noindent Ralph Kritzinger, Institut f\"{u}r Finanzmathematik und angewandte Zahlentheorie, Johannes Kepler Universit\"{a}t Linz, Altenbergerstra{\ss}e 69, A-4040 Linz, Austria. Email: ralph.kritzinger(at)jku.at \\

\noindent Lisa M. Kritzinger, Johannes Kepler Universit\"{a}t Linz, Altenbergerstra{\ss}e 69, A-4040 Linz, Austria. Email: k1255353(at)students.jku.at

\end{document}